\documentclass[12pt]{article}

\usepackage{amsmath,amssymb,amsthm}

\usepackage{float}
\usepackage[utf8]{inputenc}
\usepackage[english]{babel}
\usepackage[T1]{fontenc}
\usepackage{lmodern}
\usepackage[totalwidth=160mm, totalheight=255mm]{geometry}

\usepackage[auth-sc, affil-it]{authblk}
\usepackage{tikz}
\usepackage{tikz-3dplot}
\usetikzlibrary{calc}
\usepackage{verbatim}

\usepackage{pstricks}
\usepackage{pgf}
\usepackage{curve2e}

\newcommand {\tb}{\textcolor{blue}}

\newcommand{\bbN}{\mathbb{N}}

\newcommand{\bbR}{\mathbb{R}}

\newcommand{\bbC}{\mathbb{C}}
\newcommand{\bbD}{\mathbb{D}}
\newcommand{\Kc}{\mathcal{K}}

\newcommand{\Fc}{\mathcal{F}}
\newcommand{\Ec}{\mathcal{E}}
\newcommand{\Dc}{\mathcal{D}}

\newcommand{\M}{\mathcal{M}}

\newcommand{\Rc}{\mathcal{R}}

\theoremstyle{plain}

\newtheorem{main-theorem}{Theorem}
\newtheorem{theo}{Theorem}[section]
\newtheorem{proposition}[theo]{Proposition}
\newtheorem{lemma}[theo]{Lemma}

\newtheorem{deff}[theo]{Definition}
\newtheorem{exam}[theo]{Example}
\newtheorem{rem}[theo]{Remark}
\newtheorem{que}[theo]{Question}

\newtheorem*{clai-nn}{Claim}
\newtheorem*{theorem*}{Theorem}
\newtheorem*{definition*}{Definition}
\newtheorem*{remark*}{Remark}

\newtheorem*{schoencond*}{Sch\"onflies Condition}

\usepackage[pagewise]{lineno}%\linenumbers
\author[1]{Li Feng}
\author[2]{Jun Luo}
\author[3]{Xiao-Ting Yao} % or 15217315042@163.com}

\affil[1]{Mathematics and Computer Science, Albany State University, Albany, Georgia 31705, USA}
\affil[1]{li.feng@asurams.edu}
\affil[2]{School of Mathematics,
    Sun Yat-sen University, Guangzhou 510275, China}
\affil[2]{luojun3@mail.sysu.edu.cn}
\affil[3]{School of Mathematical Sciences, Fudan University, Shanghai 200433, China}
\affil[3]{(corresponding author): yaoxiaoting@fudan.edu.cn}

\title{\Large On Lambda Function and a Quantification of Torhorst Theorem\thanks{Supported by the Chinese National Natural Science Foundation Projects 11871483, 11771391.}}
\date{}
\begin{document}
\maketitle

\begin{abstract}
To any compact $K\subset\hat{\mathbb{C}}$ we associate a map $\lambda_K: \hat{\mathbb{C}}\rightarrow\mathbb{N}\cup\{\infty\}$,  the lambda function, so that a planar continuum $K$ is locally connected if and only if $\lambda_K(x)=0$ for all $x\in\hat{\mathbb{C}}$.
%This map is based on the finest upper semi-continuous decomposition of $K$ into sub-continua, denoted as $\Dc_K^{PC}$, such that the quotient space is a Peano compactum.
%Here a  \emph{Peano compactum} is a compact metrisable space having locally connected components of which at most finitely are of diameter greater than an arbitrary $C>0$.
We establish inequalities and a gluing lemma for the lambda functions of specific compacta. One of these inequalities comes from an interplay between the topological difficulty of a planar compactum $K$ and that of a compactum $L$ lying on the boundary of a component of $\hat{\mathbb{C}}\setminus K$.
It generalizes and quantifies the Torhorst Theorem, a fundamental result of plane topology. We also find three conditions under which this inequality becomes an equality. Under one of these conditions, such an equality provides a natural extension of Whyburn's Theorem, which  is a partial converse to Torhorst Theorem. The gluing lemma mentioned above is of its own interest and also has motivations from the study of polynomial Julia sets.

\textbf{Keywords.} \emph{Lambda Function, \ Torhorst Theorem, Gluing Lemma}
%\ Renormalizable Julia set.}

\textbf{Math Subject Classification 2010: Primary 30D40, Secondary 54D05.}
\end{abstract}

\newpage
%{\footnotesize \tableofcontents}

%%% article text:
%%%%%%%%%%%%%%%%%%%%%%%%%%%%%%%%%%%%%%%%%%%%%%%%%
%%%%%%%%%%%%%%%%%%%%%%%%%%%%%%%%%%%%%%%%%%%%%%%%%
%%%%%%%%%%%%%%%%%%%%%%%%%%%%%%%%%%%%%%%%%%%%%%%%%

\section{Lambda Function and Motivations} %{Introduction}

The existence of the core decomposition with Peano quotient for planar compacta \cite{LLY-2019} enables us to associate to each compact set $K\subset\hat{\bbC}$ a map $\lambda_K:\hat{\bbC}\rightarrow\bbN\cup\{\infty\}$, called the {\em lambda function of $K$}. This function sends all points $x\notin K$ to zero and may take a positive value for some $x\in K$. It ``quantifies'' certain aspects of the topological structure of $K$, that is more or less related to the property of being locally connected. In particular, a continuum $K\subset\hat{\bbC}$ is locally connected if and only if $\lambda_K(x)=0$ for all $x\in\hat{\mathbb{C}}$. On the other hand, if  a continuum $K\subset\hat{\bbC}$ is not locally connected at $x\in K$ then $\lambda_K(x)\ge1$; but the converse is not necessarily true.

The quantification in terms of lambda function allows us to carry out a new analysis of the topology of $K$, by computing or estimating  $\lambda_K(x)$ for specific choices of $x\in K$. In the current paper, we will investigate an interesting phenomenon that was firstly revealed in a fundamental result by Marie Torhorst, as one of the three highlights of \cite{Torhorst}. This result is often referred to as Torhorst Theorem \cite[p.106, (2.2)]{Whyburn42} and reads as follows.
\begin{theorem*}[{\bf Torhorst Theorem}]
The boundary $F$ of every complementary domain $R$ of a locally connected continuum $M\subset\hat{\mathbb{C}}$ is itself a locally connected continuum.
%If $K\subset\hat{\bbC}$ is a locally connected continuum and if $U$ is a component of $\hat{\bbC}\setminus K$ then the boundary $\partial U$ is also a locally connected continuum.
\end{theorem*}

We will obtain an inequality that includes the Torhorst Theorem as a simple case.
The inequality is about the lambda function $\lambda_K$. The function $\lambda_K$ is based on the core decomposition of $K$ with Peano quotient \cite{LLY-2019}, which is motivated by some open questions in \cite{Curry10} and extends two earlier models of polynomial Julia sets developed in \cite{BCO11,BCO13}. Those models, briefly called BCO models, provide efficient ways (1) to describe the topology of unshielded compacta, like polynomial Julia sets, and (2) to obtain specific factor systems for polynomials restricted to the Julia set. The BCO models are special cases of a more general model, working well for all planar compacta, that associates natural factor systems to the dynamics of  rational functions \cite{LLY-2019,LYY-2020}.

Recall that a {\bf Peano continuum} means the image of $[0,1]$ under a continuous map. By  Hahn-Mazurkiewicz-Sierpi\'nski Theorem
\cite[p.256, \S 50, II, Theorem 2]{Kuratowski68},  a continuum is locally connected if and only if it is a Peano continuum. On the other hand,
a {\bf Peano compactum} is defined to be a compactum having locally connected components such that  for any constant $C>0$ at most finitely many of its components are of diameter greater than $C$. Therefore, the Cantor ternary set is a Peano compactum and a Peano continuum is just a Peano compactum that is connected.  Concerning how such a definition arises from the discussions of BCO models, we refer to  \cite[Theorems 1-3]{LLY-2019}.

%Under mild assumptions, the {\bf Continuity Theorem} can be extended to the case  of conformal homeomorphism of infinitely connected circle domains \cite[Theorems 4 and 5]{LY-2020}. \begin{theorem*}[{\bf Extended Continuity Theorem}] Let $\Omega\subset\hat{\bbC}$ be a domain having at most countably many non-degenerate boundary components $P_n$ such that the sum of diameters $\sum_n{\rm diam}(P_n)$ is finite. Let $\varphi:D\rightarrow \Omega$ be a conformal homeomorphism from a circle domain $D$ onto $\Omega$. Under either of the following assumptions, \begin{itemize} \item the linear measure of $\displaystyle\partial\Omega\setminus\bigcup_nP_n$ is $\sigma$-finite, \item the linear measure of $\partial D$ is $\sigma$-finite, \end{itemize} the map $\varphi$ has a continuous extension  $\overline{\varphi}: \overline{D}\rightarrow\overline{\Omega}$ if and only if $\partial\Omega$ is a Peano compactum. \end{theorem*}

Given a compactum $K\subset\hat{\mathbb{C}}$, there exists an upper semi-continuous decomposition  of $K$ into sub-continua, denoted as $\Dc_K^{PC}$, such that (1) the quotient space is a Peano compactum and (2) $\Dc_K^{PC}$ refines every other such decomposition of $K$  \cite[Theorem 7]{LLY-2019}.
We call $\Dc_K^{PC}$ the core decomposition of $K$ with Peano quotient. The hyperspace $\Dc_K^{PC}$ under quotient topology is called the Peano model of $K$. Every $d\in\Dc_K^{PC}$ is called an {\bf atom} of $K$, or an {\bf order-one atom}, or an atom of order $1$. Every atom of an order-one atom is called an {\bf order-two atom}, and so on. Note that a compactum such as the pseudo-arc or Cantor's Teepee may have a non-degenerate atom of  order $\infty$.
%Here we note that for a compactum in $\bbR^3$, the core decomposition with Peano quotient may not exist \cite[Example 7.1]{LLY-2019}.

Considering the atoms of a compactum $K\subset\hat{\mathbb{C}}$ as its structural units, we summarize the results obtained in \cite[Theorem 7]{LLY-2019} and \cite[Theorem 1.1]{LYY-2020} in the following way.
\begin{theorem*}[{\bf Theory of Atoms}]%[{\bf\cite[Theorem 7]{LLY-2019} and \cite[Theorem 1.1]{LYY-2020}}]
Every compactum $K\subset\hat{\mathbb{C}}$ is made up of atoms; all its atoms are sub-continua of $K$ and they form an upper semi-continuous decomposition, with its quotient space being a Peano compactum, that refines every other such decomposition; moreover, for any finite-to-one open map $f:\hat{\bbC}\rightarrow\hat{\bbC}$ and any atom $d$ of $K$, each component of $f^{-1}(d)$ is an atom of $f^{-1}(K)$.
\end{theorem*}

Using the hierarchy formed by {\bf atoms of atoms}, we introduce the lambda function.
\begin{definition*}[{\bf Lambda Function}]
Given a compactum $K\subset\hat{\bbC}$. Let  $\lambda_K(x)=0$ for $x\notin K$. Let $\lambda_K(x)=m-1$ for any $x\in K$, if there is a smallest integer $m\ge1$ such that $\{x\}$ is an order-$m$ atom of $K$. If such an integer $m$ does not exist, we put $\lambda_K(x)=\infty$.
\end{definition*}
When little is known about the topology of $K$, it is  difficult to completely determine the values of $\lambda_K$. On the other hand, the level sets $\lambda_K^{-1}(n)(n\ge0)$ are ``computable''  for typical choices of $K$. In such circumstances, the lambda function $\lambda_K$ is useful in describing certain aspects of the topology of $K$. For instance, one may check the following observations: (1) a compact set $K\subset\hat{\bbC}$ is a Peano compactum if and only if $\lambda_K(x)=0$ everywhere; (2) if $K=\left\{t+\left(\sin\frac1t\right){\bf i}: 0<t\le1\right\}\cup\{s{\bf i}: |s|\le 1\}$  then $\lambda_K^{-1}(1)=\{s{\bf i}: |s|\le 1\}$ and $\lambda_K^{-1}(0)=\hat{\bbC}\setminus \lambda_K^{-1}(1)$; and (3) if $K=\left\{t+s{\bf i}: t\in[0,1], s\in\mathcal{C}\right\}$, where $\mathcal{C}$ denotes Cantor's ternary set, then $\lambda_K^{-1}(1)=K$ and $\lambda_K^{-1}(0)=\hat{\bbC}\setminus K$.

The lambda function $\lambda_K$ helps us to {\bf measure}, locally and globally, how far a compact set $K\subset\hat{\mathbb{C}}$ is from being a Peano compactum. When $K=\partial\Omega$ for a planar domain $\Omega\subset\hat{\bbC}$, which is the image of a conformal map $\varphi:\mathbb{D}\rightarrow\Omega$ of the unit disk $\mathbb{D}=\{z\in\mathbb{C}: |z|<1\}$, the study of $\lambda_K$ is particularly interesting. %Koebe's conjecture claims that $\Omega$ always admits a conformal homeomorphism $\varphi: D\rightarrow\Omega$ of some circle domain $D$. Let us restrict ourselves to the case $D=\bbD$ and recall some well known results,
In order to illustrate how $\lambda_K$ is related to the boundary behavior of $\varphi$, we recall a theorem by Fatou \cite[p.17, Theorem 2.1]{CL66}, which says that the radial limits $\lim\limits_{r\rightarrow1}\varphi\left(re^{{\bf i}\theta}\right)$ exist for all $\theta\in[0,2\pi)$, except possibly for a set of linear measure zero.
One may call the map $e^{{\bf i}\theta}\mapsto \lim\limits_{r\rightarrow1}\varphi\left(re^{{\bf i}\theta}\right)$ the {\bf boundary function} of $\varphi$, denoted  as $\varphi^b$. Sometimes we also call  $\varphi^b$ the {\bf boundary of $\varphi$}. Its domain consists of all points $e^{{\bf i}\theta}\in\partial\mathbb{D}$ such that the limit $\lim\limits_{r\rightarrow1}\varphi\left(re^{{\bf i}\theta}\right)$ exists.

Recall that the prime end at $e^{{\bf i}\theta}$ is of the first or the second type, if  the  limit $\lim\limits_{r\rightarrow1}\varphi\left(re^{{\bf i}\theta}\right)$  exists \cite[p.177, Theorem 9.7]{CL66}. %In the special case that $\Omega$ has no prime end of the second type, we may obtain a continuous map $\hat{\varphi}=\varphi\cup\varphi^b$ by gluing together $\varphi$ and its boundary $\varphi^b$. The domain of $\hat{\varphi}$ excludes every $\zeta\in\partial\bbD$ at which the radial limit does not exist.
%  \cite[p.183, Theorem 9.11]{CL66}.
By Carath\'eodory's {\bf Continuity Theorem} \cite[p.18]{Pom92}, the boundary function $\varphi^b$ is defined well and continuous on the whole unit circle $\partial\bbD$ if and only if the boundary $\partial\Omega$ is a Peano continuum. By \cite[p.17,Theorem 2.1]{CL66}, if $\partial\Omega$ is not a Peano continuum then $\varphi^b$ is a member of $L^\infty(\partial\bbD)\setminus C(\partial\bbD)$. Therefore, we are curious about two types of quantities: (1) those describing how far $\varphi^b$ is from being continuous and (2) those  measuring how far $\partial\Omega$ is from being locally connected.

The first type concerns the asymptotic of $\varphi(z)$ for $z\in\bbD$ as $|z|\rightarrow 1$. The second type concerns the topology of $\partial\Omega$.
%These quantities are of some help, when we try to describe the interplay between the asymptotic of $\varphi(z)$ and the topology of $\partial\Omega$.
The lambda function $\lambda_{\partial\Omega}$ gives rise to a quantity of the second type. This function vanishes everywhere if and only if $\varphi^b$ belongs to $C(\partial\bbD)$. Among others, the questions below are of some interest: {\em What does it mean if $\lambda_{\partial\Omega}(x)\le 1$ for all $x$? Can we say something about $\lambda_{\partial\Omega}$ when $\varphi^b$ has finitely many discontinuities?}

In particular, one may consider the case that $\varphi(\mathbb{D})$ coincides with the complement of the Mandelbrot set $\M$. The set $\mathcal{M}$ consists of all the complex parameters $c$ such that the Julia set of  $f_c(z)=z^2+c$ is connected. Its local connectedness is still an open question. % The conjecture is YES.
Due to the deep works on quadratic polynomials, we know that $\M$ is locally connected at many of its points, such as the Misiurewicz points and those lying on the boundary of a hyperbolic component. From these known results arises  a natural question: {\bf Is it true that $\lambda_\M(x)=0$ for all $x$ at which $\M$ is known to be locally connected?} Another  of some interest is, {\bf Can we find some upper bound for the lambda function $\lambda_\M$?} In particular, {\bf Can we show that $\lambda_\M(x)\le1$ for all $x$?}
Very recently, Yang and Yao \cite{YY-2020} show that $\lambda_{\mathcal{M}}(x)=0$ for all points $x$ lying on the boundary of a hyperbolic component. More studies on questions of a similar nature can be expected.

%Among all topological spaces, those that are compact and Hausdorff provide a stage on which a direct bridge is demonstrated, that connects the study of topology and that of algebra. In deed, two such spaces $X, Y$ are topologically equivalent (as homeomorphic spaces) if and only if the two algebras of continuous complex functions $C(X), C(Y)$ are algebraically equivalent (as isomorphic algebras). The ``only if'' part is direct, the ``if'' part may be deduced from Gelfand-Naimark Theorem \cite[p.14, Theorem 7.1]{Gamelin69}. One may refer to \cite[p.57, Theorem 4.9]{Gillman-Jerison76} for the same result concerning rings of continuous real functions.

\section{Main Results}

In the current paper, we want to analyze $\lambda_K, \lambda_L$ for compacta $L\subset K\subset\hat{\mathbb{C}}$, that satisfy specific  properties. Our analysis is connected to very basic results of topology, such as the Torhorst Theorem and the gluing lemma for continuous functions.

We will extend and quantify the  Torhorst Theorem by an inequality, stating that the lambda function $\lambda_K$ of any compactum $K\subset\hat{\mathbb{C}}$ is an {\bf upper bound} of $\lambda_{\partial U}$ for any  complementary component $U$ of $K$.
%Moreover, there are several special sub-cases, under which we actually have equalities in terms of the lambda function. One of those equalities generalizes a theorem by Whyburn \cite[p.113, (4.4)]{Whyburn42}, which is  a partial converse of the Torhorst Theorem.
To this end, let us examine how the atoms of $K$ are related to those of any compactum  $L$ lying on the boundary of a component $U$ of $\hat{\bbC}\setminus K$.
\begin{main-theorem}\label{compare_atoms}
Given a compact $K\subset\hat{\mathbb{C}}$ and a component $U$ of \ $\hat{\bbC}\setminus K$. If $L\subset\partial U$ is a compactum then every atom of $L$ lies in a single atom of $K$. Consequently, every atom of $L$ lies in a single atom of $\partial U$. \end{main-theorem}

With Theorem \ref{compare_atoms}, we compare the lambda functions  $\lambda_K,\lambda_L$ and obtain the following.

\begin{main-theorem}%[Lambda Inequality]
\label{lambda_inequality}
Given a compactum $K\subset\hat{\mathbb{C}}$ and  a component $U$ of \ $\hat{\bbC}\setminus K$. If $L\subset\partial U$ is a compactum then $\lambda_L(x)\le \lambda_K(x)$ for all $x\in\hat{\bbC}$.
\end{main-theorem}

Let $\mathcal{A}$ consist of the components of $\hat{\mathbb{C}}\setminus K$. The  {\bf envelope function} of $K$ is defined as
%\begin{equation}\label{lambda_equation}
$\displaystyle \tilde{\lambda}_K(x)=\sup\limits_{L\subset\partial U, U\in\mathcal{A}}\lambda_L(x)\ \left(x\in\hat{\bbC}\right).$
%\end{equation}
By Theorems \ref{compare_atoms} and \ref{lambda_inequality}, we have $\tilde{\lambda}_K(x)=\sup\limits_{U\in\mathcal{A}}\lambda_{\partial U}(x)$  and $\tilde{\lambda}_K(x)\le\lambda_K(x)$ for all $x\in\hat{\bbC}$. From now on we call $\tilde{\lambda}_K(x)\le\lambda_K(x)\left(x\in\hat{\mathbb{C}}\right)$  the {\bf\em lambda inequality}, which is also written as $\tilde{\lambda}_K\le\lambda_K$. If $K$ is a Peano compactum then $\lambda_K$ and hence $\lambda_L$ vanish everywhere, for any compactum $L$ lying on the boundary of any component $U$ of $\hat{\mathbb{C}}\setminus K$. That is to say, each compactum $L\subset\partial U$ is a Peano compactum. This includes as a special sub-case the Torhorst Theorem, in which  $K$ is assumed to be a Peano continuum and  $L=\partial U$.

\begin{rem}
In the above setting, the boundary $\partial K$  may not be locally connected, even if it is a continuum.
%In such a case, we have $\lambda_K(x)\equiv0$ and $\lambda_{\partial K}(x)\neq0$.
See \cite[Example 3.2]{Luo07} or Example \ref{bd_larger} of this paper. In Example \ref{bd_smaller}, we will construct a continuum $L$ such that the range of $\lambda_L-\lambda_{\partial L}$ is $\{-1,0,1\}$.  %From those examples, we see that the connection between $\lambda_K$ and $\lambda_{\partial K}$ is not clear, for a general compactum $K\subset\hat{\bbC}$.
\end{rem}

Now we consider compacta $K\subset\hat{\mathbb{C}}$ for which the {\bf Lambda Equality} $\tilde{\lambda}_K=\lambda_K$ holds, so that $\tilde{\lambda}_K(x)=\lambda_K(x)$ for all $x\in\hat{\bbC}$.
Let us start from a simple example.
%showing that for some component $U$ of $\hat{\mathbb{C}}\setminus K$ the inequality $\tilde{\lambda}_K(x)<\lambda_K(x)$ may hold for all $x\in\partial U$.
\begin{exam}\label{why_E_compactum}
Let $U$ be the complement of the unit square $\{a+b{\bf i}: 0\le a,b\le 1\}$. Let $K$ be the compactum consisting of $\partial U$ and an infinite sequence of squares (only the boundary) of side length $<1$. These squares are centered at $0.5+0.5{\bf i}$ and converge to $\partial U$ under Hausdorff distance. Then $\lambda_K(x)-\tilde{\lambda}_K(x)=\left\{\begin{array}{ll}1& x\in\partial U\\ 0&\text{otherwise}.\end{array}\right.$
%Then $\lambda_K(x)=1$ for all $x\in\partial U$, while  $\tilde{\lambda}_K(x)$ vanishes everywhere.
%See Figure \ref{seq-squares} for a simplified depiction of $K$.

\begin{figure}[ht] \vskip -0.5cm \begin{center} \begin{tikzpicture}[x=0.2cm,y=0.2cm,scale=0.45] \draw[black,thick] (0,0) -- (0,32) -- (32,32) -- (32,0) -- (0,0);  \foreach \j in {3,...,6} {    \draw[gray, thick] (2*\j,32-2*\j) -- (32-2*\j,32-2*\j) -- (32-2*\j,2*\j) -- (2*\j,2*\j) --(2*\j,32-2*\j); }     \foreach \j in {2,...,5} {\fill[gray] (\j,\j) circle(0.5ex); \fill[gray] (32-\j,32-\j) circle(0.5ex); \fill[gray] (\j,32-\j) circle(0.5ex); \fill[gray] (32-\j,\j) circle(0.5ex); }    \draw(0,1) node[left]{$0$}; \draw(32,1) node[right]{$1$}; \draw(0,31) node[left]{${\bf i}$}; \end{tikzpicture} \end{center} \vskip -1.0cm \caption{A simplified depiction of $K$ and the small squares.}\label{seq-squares}\vskip -0.25cm \end{figure}
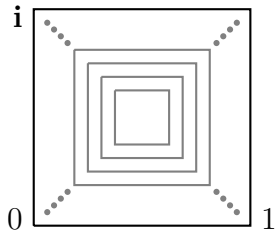
\end{exam}

If $K$ is given in Example \ref{why_E_compactum} then $\hat{\mathbb{C}}\setminus K$ has infinitely many components of diameter greater than $1-\epsilon$ for some constant $\epsilon\in(0,1)$. Borrowing the idea of {\bf $E$-continuum} \cite[p.113]{Whyburn42}, we define an {\bf $E$-compactum} to be a planar compactum such that for any constant $\varepsilon>0$ its complement has at most finitely many components of diameter $>\varepsilon$. Unfortunately, the condition of $E$-compactum alone is still not sufficient for the {\bf Lambda Equality} $\tilde{\lambda}_K=\lambda_K$.  See Examples \ref{E-compactum} and \ref{finite-comp}.

The theorem below gives three conditions under which the Lambda Equality holds.

\begin{main-theorem}\label{equality-case}
Given a compactum $K\subset\hat{\mathbb{C}}$, the Lambda Equality  $\tilde{\lambda}_K=\lambda_K$ holds if one of the following conditions is satisfied:

(i) $K$ is an $E$-compactum such that the envelope function $\tilde{\lambda}_K(x)$ vanishes everywhere;

(ii)  $K$ is an $E$-compactum whose complementary components have disjoint closures.

(iii) $K$ is a partially unshielded compactum.
\end{main-theorem}

\begin{rem}
In (i) and (ii) of Theorem \ref{equality-case}, the assumption  that $K$ is an $E$-compactum  can not be removed. We may set $K=[0,1]\!\times\![0,{\bf i}]\setminus\left(\bigcup\limits_1^\infty R_n\right)$, with $R_n=\left(\frac{1}{3n},\frac{2}{3n}\right)\times\left(\frac13{\bf i},\frac23{\bf i}\right)$. If $W=\hat{\mathbb{C}}\setminus[0,1]\!\times\![0,{\bf i}]$ then $\partial W\cap\partial R_n=\emptyset$ for all $n\ge1$ and  $\partial R_n\cap \partial R_m=\emptyset$ for $n\ne m$. See Figure \ref{non-E} for a simple depiction of $K$.
The continuum $K$ is not an $E$-continuum but it satisfies the other assumptions in (i) and (ii) of Theorem \ref{equality-case}. It has exactly one non-degenerate atom, the segment $\left[\frac13{\bf i},\frac23{\bf i}\right]$. Thus $\lambda_K(x)-\tilde{\lambda}_K(x)=\left\{\begin{array}{ll}1& x\in\left[\frac13{\bf i},\frac23{\bf i}\right]\\ 0&\text{otherwise}.\end{array}\right.$
%See Figure \ref{non-E}.
%Therefore, we have \[\lambda_K(x)=\left\{\begin{array}{ll}1&x\in\left[\frac13{\bf i},\frac23{\bf i}\right]\\0& {otherwise}.\end{array}\right.\]

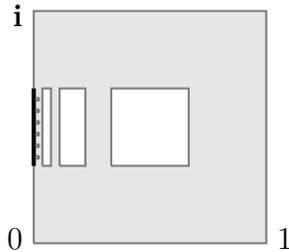
\begin{figure}[ht]
\vskip -0.75cm
\begin{center}
\begin{tikzpicture}[x=5cm,y=5cm,scale=0.618]
\fill[gray!20,thick] (0,0) -- (0,1) -- (1,1) -- (1,0) -- (0,0);
\draw[gray,thick] (0,0) -- (0,1) -- (1,1) -- (1,0) -- (0,0);
\draw[black, ultra thick] (0,1/3) -- (0,2/3);

\foreach \j in {1,...,3}
{
 \fill[white] (1/3^\j,1/3) -- (2/3^\j,1/3) -- (2/3^\j,2/3) -- (1/3^\j,2/3) --(1/3^\j,1/3);
    \draw[gray, thick] (1/3^\j,1/3) -- (2/3^\j,1/3) -- (2/3^\j,2/3) -- (1/3^\j,2/3) --(1/3^\j,1/3);
}
\foreach \j in {1,...,6}
{\fill[gray] (1/60,0.32+0.05*\j) circle(0.3ex);
}

\draw(0,0.02) node[left]{$0$};
\draw(1,0.02) node[right]{$1$};
\draw(0,0.98) node[left]{${\bf i}$};
%\draw(1/2,2.75/5) node[below]{$R_1$};
%\draw(1/6,2.75/5) node[below]{$R_2$};

%\draw (1.2,0.5) node[right] {$\lambda_K(x)=\left\{\begin{array}{ll}1&x\in\left[\frac13{\bf  i},\frac23{\bf i}\right]\\ & \\ 0& {otherwise}.\end{array}\right.$};
\end{tikzpicture}
\end{center}
\vskip -0.95cm
\caption{A depiction for $K$ and some of the rectangles.}\label{non-E}
\vskip -0.25cm
\end{figure}
\end{rem}

Note that the Lambda Equality may not hold for an $E$-compactum  $K\subset\hat{\mathbb{C}}$, even if it  has finitely many complementary components. See Example \ref{finite-comp}. Also notice that the Lambda Equality under condition (i) implies the theorem below. This extends Whyburn's Theorem \cite[p.113, (4.4)]{Whyburn42}, which says that {\em an $E$-continuum is a Peano continuum if and only if the boundary of any of its complementary components is a Peano  continuum}.
\begin{theorem*}[{Extended Whyburn's Theorem}] An $E$-compactum is a Peano compactum if and only if the boundary of any of its complementary components is a Peano compactum.
\end{theorem*}

Theorem \ref{lambda_inequality} addresses how $\lambda_K$ and $\lambda_L$ are related when $L$ lies on the boundary of a component of $\hat{\bbC}\setminus K$. There are other choices of  planar compacta $K\supset L$ so that $\lambda_K$ and $\lambda_L$ are intrinsically related. A typical situation happens, if the common part of $\overline{K\setminus L}$ and $L$ is  a finite set.

\begin{main-theorem}\label{gluing_lemma}
If $K\supset L$ are planar compacta such that $\overline{K\setminus L}$ intersects $L$ at finitely many points then $\lambda_K(x)=\max\left\{\lambda_{\overline{K\setminus L}}(x),\lambda_L(x)\right\}$ for all $x$.  \end{main-theorem}

Setting $A=\overline{K\setminus L}$, we can infer that that $\lambda_K(x)$ coincides with $\lambda_A(x)$ for $x\in A\setminus L$ and with $\lambda_L(x)$ for $x\in L\setminus A$, equals $\max\left\{\lambda_{A}(x),\lambda_L(x)\right\}$ for $x\in A\cap L$, and vanishes for every $x\notin(A\cup L)$. Therefore, we have.

\begin{theorem*}[{Gluing Lemma for Lambda Functions}]\label{gluing_lemma_1}
If in addition $\lambda_A(x)=\lambda_L(x)$ for all $x\in A\cap L$ then $\lambda_K(x)=\lambda_{A\cup L}$ may be obtained by gluing $\lambda_A$ and $\lambda_L$, in the sense that
\begin{equation}\label{form-1}
\lambda_{A\cup L}(x)=\left\{\begin{array}{ll}\lambda_A(x)& x\in A\\ \lambda_L(x)& x\in L\\ 0& {otherwise.}\end{array}\right.
\end{equation}
\end{theorem*}
\begin{rem} %The above {\bf gluing lemma} for lambda functions provides conditions under which we may form the lambda function $\lambda_K$ by ``gluing together'' $\lambda_{A}$ and $\lambda_L$.
The formation in Equation (\ref{form-1}) is similar to the one illustrated in the well known gluing lemma for continuous maps. See for instance \cite[p.69, Theorem (4.6)]{Armstrong}.
In certain situations, Theorem \ref{gluing_lemma} helps us to analyze questions concerning local connectedness of polynomial Julia sets. See Question \ref{small_julia}. However,
the case that $A\cap L$ is an infinite set is more involved. In Theorem \ref{baby_M}, we will extend Theorem 4 to such a case under additional assumptions. This extension allows one to choose $K$ to be the Mandelbrot set and $L$ the closure of a hyperbolic component. For concrete choices of $A$ and $L$ so that Equation \ref{form-1} does not hold, we refer to Examples \ref{cantor_combs}, \ref{brooms} and \ref{cup-fs}.
%This extension allows one to choose $K$ to be the Julia set of {\bf a renormalizable polynomial $f$} and $L$ the {\bf small Julia set}, so that $A=\overline{K\setminus L}$. In such a situation, we can find a decreasing sequence of Jordan domains  $\{U_n\}$ such that (1) $K\cap \partial U_n$ is a finite set for all $n\ge1$ and (2) $\displaystyle L=\bigcap_{n=1}^\infty \overline{U_n}$. Among others, one may see \cite{Jiang00} for the construction of $U_n$ when $f$ is a quadratic polynomial. In such a case, there is a further question of some interest that concerns the compacta $L_n=K\cap\overline{U_n}$ for $n\ge1$: {\bf Is it true that $\lambda_L(x)=\lim\limits_{n\rightarrow\infty}\lambda_{L_n}$ for all $x$?}
%Similar discussions also work well for more general choices of $K\supset L$ such that $K\setminus L$ has countably many components, if each of these components has a single limit point on $L$ and if for any $C>0$ only finitely many of them are of diameter greater than $C$. See Theorem \ref{baby_M} for such discussions.
%This is the case when $K$ is  the Mandelbrot set $\M$ and $L$ is a baby. Many results are obtained, centering around the topology of $\M$. In particular, the smallest closed equivalence is obtained, whose restriction to $\mathbb{Q}/\mathbb{Z}$ is given by the equivalence $\sim^\mathbb{Q}_\M$. See for instance \cite{Douady93} and the references therein.
\end{rem}

The other parts of this paper are arranged as follows. Section \ref{proof-c} is devoted to the proofs for Theorems \ref{compare_atoms} to \ref{lambda_inequality}. Section \ref{equality} gives a proof for  Theorem \ref{equality-case}. In Section \ref{glue} we firstly prove Theorem \ref{gluing_lemma} and then continue to establish Theorem \ref{baby_M}.
%that deals with a more general case of ``gluing two lambda functions $\lambda_A,\lambda_L$'' to obtain the lambda function $\lambda_{A\cup L}$.
Section \ref{examples} gives examples.
%In particular, section \ref{gluing_negative} provides two examples of planar continua $K\supset L$ that do not satisfy the assumptions in Theorem \ref{gluing_lemma}; section \ref{E-bad} constructs an $E$-continuum $K$ for which the Lambda Equality does not hold; section \ref{k_vs_bd} gives a concrete continuum $K$ such that  $\lambda_K(x)>\lambda_{\partial K}(x)$ for certain points $x\in\partial K$. The  continuum $K$  given in section \ref{k_vs_bd} contains a sub-continuum $L$, such that $\lambda_L(x)\equiv0$ and that $\lambda_{\partial L}(x)=1$ for certain points $x\in\partial L$.
%At the end there is  an appendix, in which we introduce a geometric graph $\Lambda_K$ for any planar compactum $K$. Such a graph demonstrates the relative locations of the level sets $\lambda_K^{-1}(m)$ for $m\ge0$ thus is useful when we compute or estimate the values of $\lambda_K(x)$ for specific choices of $x\in K$.

%Application ID: AA0080SNCA  Date: 07-JUN-2018

\section{The Lambda Inequality}\label{proof-c}

In this section we prove Theorems \ref{compare_atoms} and \ref{lambda_inequality}.

We will study relations on compacta $K\subset\hat{\mathbb{C}}$. Such a relation is considered as a subset of the product space $K\times K$ and  is said to be {\bf closed} if it is closed in $K\times K$. Given a relation $\Rc$ on $K$, we call $\Rc[x]=\{y\in K: \ (x,y)\in\Rc\}$ the fiber of $\Rc$ at $x$. We mostly consider  closed relations $\Rc$ that are reflexive and symmetric, so that for all $x,y\in K$ we have (1) $x\in \Rc[x]$ and (2) $x\in\Rc[y]$ if and only if $y\in\Rc[x]$. For such a relation, the {\em iterated relation} $\Rc^2$ is defined naturally so that
$\displaystyle y\in\Rc^2[x]$ if and only if there exist $z\in K$ with $(x,z), (z,y)\in \Rc$.

Recall that the {\bf Sch\"onflies relation} on a planar compactum $K$ is a reflexive symmetric relation. Under this relation, two points $x_1, x_2$ are related provided that either $x_1=x_2$ or there exist two disjoint Jordan curves $J_i\ni x_i$ such that  $\overline{U}\cap K$ has infinitely many components $P_n$, intersecting $J_1$ and $J_2$ both, whose limit under Hausdorff distance contains $\{x_1,x_2\}$. Here $U$ is the component of \ $\hat{\bbC}\setminus(J_1\cup J_2)$ with $\partial U=J_1 \cup J_2$.

Given a compactum $K\subset\hat{\bbC}$,  denote by $R_K$ the  Sch\"onflies relation on $K$ and by $\overline{R_K}$ the closure of $R_K$. We also call $\overline{R_K}$ the {\bf closed Sch\"onflies relation}.
Let $\Dc_K$ be the finest upper semi-continuous decompositions of $K$ into sub-continua that splits none of the fibers $R_K[x]$. Then $\Dc_K$ coincides with $\Dc_K^{PC}$, the core decomposition of $K$ with Peano quotient \cite[Theorem 7]{LLY-2019}. Therefore the elements of $\mathcal{D}_K$ are the (order-one) atoms of $K$.

Every fiber $\overline{R_K}[x]$ is a continuum. See %\cite[Theorem 5.1]{LLY-2019} or
\cite[Theorem 1.4]{LYY-2020}. However, the compactness and connectedness of the fibers of $R_K$ remain open.
Moreover, in order that a point $y\ne x$ lies  in $\overline{R_K}[x]$ it is necessary and sufficient that for small enough  $r>0$ the difference $K\setminus(B_r(x)\cup B_r(y))$ has infinitely many components that intersect each of the circles $\partial B_r(x)$ and $\partial B_r(y)$. See  \cite[Theorem 1.3]{LYY-2020}.
The lemma below relates the fibers of $\overline{R_K}^2$ to those of $\overline{R_L}$, where $L$ is a compact subset of $K$ satisfying certain properties.

\begin{lemma}\label{key-lemma}
Given a compactum $K\subset\hat{\mathbb{C}}$ and  a component $U$ of \ $\hat{\bbC}\setminus K$. If $L\subset\partial U$ is compact then $\overline{R_L}[x]\subset \overline{R_K}^2[x]$  for any $x\in L$. %Consequently, every fiber of $\overline{R_L}$ is contained in an atom of $K$.
\end{lemma}
\begin{proof}
To obtain the containment $\overline{R_L}[x]\subset \overline{R_K}^2[x]$ for any given $x\in L$,  we may fix an arbitrary point $y\in\overline{R_L}[x]\setminus\{x\}$ and consider the annulus $A_n=\hat{\mathbb{C}}\setminus\left(B_{1/n}(x)\cup B_{1/n}(y)\right)$, for any integer $n\ge1$ such that $\overline{B_{1/n}(x)}\bigcap\overline{B_{1/n}(y)}=\emptyset$. Here $B_{1/n}(x)$ and $B_{1/n}(y)$ are open disks with radius $1/n$ under spherical distance, that are respectively centered at $x$ and $y$. By \cite[Theorem 1.3]{LYY-2020}, $A_n\cap L$ has infinitely many components intersecting both $\partial B_{1/n}(x)$ and $\partial B_{1/n}(y)$. So we can find an infinite sequence $\{P_i\}$ of such components that converge to some continuum $P_\infty$ under Hausdorff metric.
%Let $P_0$ be the component of $A_n\cap L$ containing $P_\infty$.

Since $x,y\in L\subset\partial U$, we can find an open arc $\alpha_0\subset U$ that connects a point on $\partial B_{1/n}(x)$ to one on $\partial B_{1/n}(y)$. Going to an appropriate sub-arc, if necessary, we may assume that $\alpha\subset A_n$. Then, we may slightly thicken the closed arc $\overline{\alpha}$ and obtain a topological disc $\alpha^*\subset A_n$, satisfying $\alpha^*\cap K=\emptyset$. From this we see that $\overline{A_n\setminus\alpha^*}$ is homeomorphic to $[0,1]^2$. We will obtain the following.

{\bf Claim}. $P_\infty$ contains two points $u_n\in \partial B_{1/n}(x), v_n\in \partial B_{1/n}(y)$ with $v_n\in\overline{R_K}^2[u_n]$.

The flexibility of the large enough integers $n$ ensures that $\lim\limits_nu_n=x$ and $\lim\limits_nv_n=y$. Since $\overline{R_K}^2$ is a closed relation, we surely obtain $y\in \overline{R_K}^2[x]$.  This completes our proof. Thus, the remaining issue is to verify the above claim.

As $\overline{A_n\setminus\alpha^*}$ is a topological disc, we consider it to be the unit square $[0,1]^2$. Moreover, we may represent by $[0,1]\times\{1\}$  the arc $l_1=\overline{A_n\setminus\alpha^*}\cap\partial B_n(x)$ and by $[0,1]\times\{0\}$ the arc $l_2=\overline{A_n\setminus\alpha^*}\cap\partial B_n(y)$. Fix any point $z$ in $P_\infty\cap(0,1)^2$. For any $r>0$ that is small, let $W_r$ denote the open rectangle centered at  $z$ with diameter $r$.
Since $P_i\rightarrow P_\infty$ under Hausdorff distance we may assume that every $P_i$ intersects $W_r$ and lies in $[0,1]^2$, which from now on represents $\overline{A_n\setminus\alpha^*}$.
%In particular, for all $i\ge1$ with $P_i\cap W_r\ne\emptyset$ we pick a point $a_i$ in $P_{2i+1}\cap W_r$.
See Figure \ref{key}. % for relative locations of $z, l_1,l_2$ and $W_r$.
\begin{figure}[ht]
\vskip -0.5cm
\center{
\begin{tikzpicture}[scale=0.8,x=1.618cm, y=0.618cm]
\draw(-2,0)--(-2,7);
\draw(-2,0)--(7,0)node[below]{$l_2\subset \partial B_n(y)$};
\draw(-2,7)--(7,7)node[above]{$l_1\subset \partial B_n(x)$};
\draw(7,0)--(7,7);
\draw(2,0)--(2,7)node[above]{$P_\infty$};
%\draw(2,3.5) node[left]{$z$};
\fill(2,3.5)circle(2pt);
\draw[blue,thick](-1,2)--(5,2) -- (5,5) -- (-1,5) -- (-1,2);
\draw(2,3.5) node[left]{$W_r\ni z$};
\draw(3,0)--(3,7)node[above]{$P_{2i+1}$};
\draw(4.5,0)--(4.5,7)node[above]{$P_{2i-1}$};
\draw[dashed](3.75,0)--(3.75,7);
\draw (3.75,0)node[below]{$P_{2i}$};
\draw(3.75,3) node[left]{$a_i$};
\fill(3.75,3)circle(2pt);
\draw(4,4) node[right]{$b_i$};
\fill(4,4)circle(2pt);
\draw[red](4.25,5)--(4.25,7); \draw(4.3,6) node[left]{$\beta_i$};
\draw[red](4.25,2)--(4.25,0); \draw(4.3,1) node[left]{$\beta_i$};
%\draw[red](2.3,2)--(2.3,0);
\end{tikzpicture}
}\vskip -0.75cm
\caption{Relative locations of $z,l_1,l_2,W_r, P_{2i-1}, P_{2i}, P_{2i+1}$ and $a_i, b_i$.}\label{key}
\vskip -0.25cm
\end{figure}
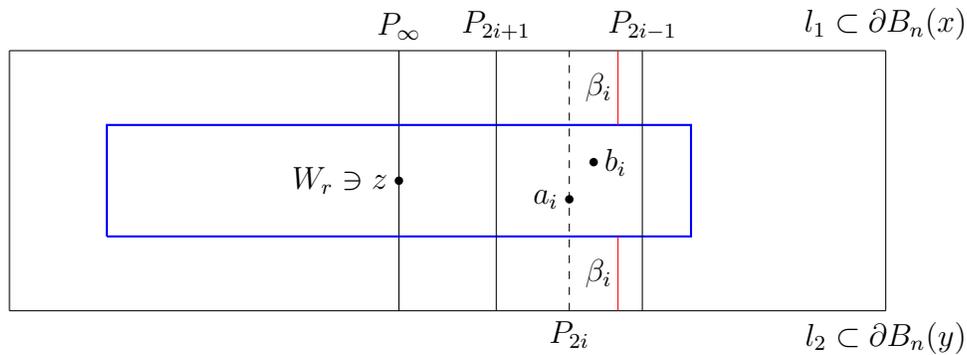

Recall that $[0,1]^2\setminus P_\infty$ has two components, one containing $\{1\}\times[0,1]$ and the other $\{0\}\times[0,1]$. One of these components contains infinitely many $P_i$. Without losing generality we may assume that every $P_i$ lies in the one containing $\{1\}\times[0,1]$, denoted  $V$. Thus $P_i$ can be connected to $\{1\}\times[0,1]$ by an arc in $[0,1]^2$ that does not intersect $P_\infty$. Moreover, rename $P_i(i\ge1)$ so that every $P_i$ can be connected to $\{1\}\times[0,1]$ by an arc in $[0,1]^2$ that does not intersect $P_j$ for $j\ge i+1$. Therefore, each $P_i$ is ``to the right of'' $P_{i+1}$.

For all $i\ge1$ let $V_i$ be the unique component of $\hat{\bbC}\setminus\left(P_{2i-1}\cup P_{2i+1}\cup l_1\cup l_2\right)$ whose boundary intersects each of $l_1$,  $l_2$, $P_{2i-1}$ and $P_{2i+1}$. Then $P_{2i}\subset \overline{V_i}$ for $i\ge1$. For the previously given point $z$ in $P_\infty\cap(0,1)^2$, we can find for each $i\ge1$ a point  $a_i\in P_{2i}\cap W_r$ such that $\lim\limits_{i\rightarrow\infty}a_i=z$. Since $P_{2i}\subset L\subset \partial U$, we further find a point $b_i\in (W_r\cap V_i\cap U)$ for every $i\ge1$, such that the distance between $a_i$ and $b_i$ converges to zero as $i\rightarrow\infty$. Check Figure \ref{key} for relative locations of $a_i\in P_{2i+1}$ and $b_i\in(W_r\cap V_i\cap U)$.

Now, we may find arcs $\alpha_i\subset U$ for each $i\ge1$ that starts from a fixed point $b_0\in U$ and ends at $b_i$. Let $c_i$ be the last point on $\alpha_i$ that leaves $\partial[0,1]^2$. Let $d_i$ be the first point on $\alpha_i$  after $c_i$ at which $\alpha_i$ intersects $\partial W_r$. Clearly, we have $c_i\in(l_1\cup l_2)$. Let $\beta_i$ be the sub-arc of $\alpha_i$ from $c_i$  to $d_i$. Check Figure \ref{key} for a rough depiction of two possible locations for $\beta_i$. Then $\beta_i$ and $\beta_j$ for $i\ne j$ are contained in distinct components of $\mathcal{A}_r\setminus L$, where $\mathcal{A}_r=[0,1]^2\setminus W_r$ is topologically a closed annulus.
Since $L\subset K$ and $K\cap U=\emptyset$, the arcs $\beta_i$ and $\beta_j$ for $i\ne j$ are contained in distinct components of $\mathcal{A}_r\setminus K$.

Let $x_n$ be the only point on $l_1\cap P_\infty$ such that the right piece of $l_1\setminus\{x_n\}$ does not intersect $P_\infty$. Let $y_n$ be the point on $l_2\cap P_\infty$ such that the right piece of $l_2\setminus\{y_n\}$ does not intersect $P_\infty$. The sequence $\{c_i\}$ then has a limit point in $\{x_n,y_n\}$. We may assume that $z_r=\lim\limits_{i\rightarrow\infty}d_i$ for some point $z_r\in\partial W_r$. Since $\partial[0,1]^2$ and $\partial W_r$ are disjoint Jordan curves, from the choices of $x_n, y_n$ and $z_r$ we can infer that either $(x_n,z_r)\in R_K$ or  $(y_n,z_r)\in R_K$. The flexibility of $r>0$ then leads to the inclusion $z\in\left(\overline{R_K}[x_n]\cup\overline{R_K}[y_n]\right)$.

%Let $Q_n=P_\infty\cap(0,1)^2\cap\overline{R_K}[x_n]$ and $Q_n'=P_\infty\cap(0,1)^2\cap\overline{R_K}[y_n]$. Then $E_n=Q_n\cup\left(l_1\cap P_\infty\right)$ and $F_n=Q_n'\cup\left(l_2\cap P_\infty\right)$
Now consider the two closed sets $E_n=P_\infty\cap \left(\overline{R_K}[x_n]\cup l_1\right)$ and $F_n=P_\infty\cap \left(\overline{R_K}[y_n]\cup l_2\right)$, which satisfy $P_\infty=E_n\cup F_n$. From the connectedness of $P_\infty$ we see that $E_n\cap F_n\ne\emptyset$. Clearly, each point $w\in (E_n\cap F_n)$ necessarily falls into one of the following cases:
\begin{itemize}
\item[(1)] $w$ lies in $l_1\subset\partial B_{1/n}(x)$ and belongs to $\overline{R_K}[y_n]$,
\item[(2)] $w$ lies in $l_2\subset\partial B_{1/n}(y)$  and belongs to $\overline{R_K}[x_n]$,
\item[(3)] $w\notin(l_1\cup l_2)$ and it lies in $\overline{R_K}[x_n]\cap\overline{R_K}[y_n]\cap(0,1)^2$.
\end{itemize}
In case (1) we set $u_n=w,v_n=y_n$; in case (2) we set $u_n=x_n, v_n=w$; in case (3) we set $u_n=x_n, v_n=y_n$. Then, in cases (1) and (2) we have $v_n\in\overline{R_K}[u_n]\subset\overline{R_K}^2[u_n]$; and in case (3) we will have $v_n\in\overline{R_K}^2[u_n]$. This verifies the claim and completes our proof.
\end{proof}

With Lemma \ref{key-lemma}, we are well prepared to prove Theorems \ref{compare_atoms} and \ref{lambda_inequality} as follows.

\begin{proof}[{\bf Proof for Theorem \ref{compare_atoms}}]
%The later part of Theorem \ref{compare_atoms} follows from the former part.
Since $U$ is also a complementary component of $\partial U$, we only verify that every atom of $L$ is contained in a single atom of $K$.

To this end, let $\Dc_L^\#$ consist of all those continua that are each a component of $d^*\cap L$ for some $d^*\in\Dc_K$. By \cite[p.44, Theorem 3.21]{Nadler92} and \cite[p.278, Lemma 13.2]{Nadler92}, we see that $\Dc_L^\#$ is an upper semi-continuous decomposition of $L$. As  every fiber of $\overline{R_K}^2$ is entirely contained in a single element of $\Dc_K$, by Lemma \ref{key-lemma} we know that every fiber $\overline{R_L}[z]$ is entirely contained in a single element of $\Dc_L^\#$. This implies that $\Dc_L^\#$ is refined by $\Dc_L$. In other words, every atom of $L$ is entirely contained in a single atom of $K$.
\end{proof}

\begin{proof}[{\bf Proof for Theorem \ref{lambda_inequality}}]
To obtain $\lambda_L(x)\le\lambda_K(x)$ for all $x$, we only need to consider the points $x\in L$. With no loss of generality, we may assume that $\lambda_K(x)=m-1$ for some integer $m\ge1$. That is to say, there exist strictly decreasing continua $d_1^*\supset d_2^*\supset\cdots\supset d_{m}^*=\{x\}$ such that $d_1^*$ is an atom of $K$ and $d_{i+1}^*$ an atom of $d_i^*$ for $1\le i\le m-1$. Here we may have $m=1$. By Theorem \ref{compare_atoms}, the atom of $L$ containing $x$, denoted as $d_1$, is a subset of $d_1^*$. Since $d_1\subset d_1^*$ also satisfies the assumptions of  Theorem \ref{compare_atoms}, we can infer that the atom of $d_1$ containing $x$, denoted as $d_2$, is a subset of $d_2^*$. Repeating the same argument for $m$ times, we obtain for $1\le i\le m$ an order-$i$ atom $d_i$ of $L$ with $d_i\subset d_i^*$. Here we have $d_{m}=\{x\}$ and hence $\lambda_L(x)\le m=\lambda_K(x)$.
\end{proof}
\begin{rem}
In the proof for Theorem \ref{lambda_inequality}, we know that $U$ is a component of $\hat{\bbC}\setminus K$ and $L\subset\partial U$. Therefore, in the same way we can show that $\lambda_L(x)\le\lambda_{\partial U}(x)$ for all $x$. From this we can infer that
$\sup\limits_U\lambda_{\partial U}(x)=\tilde{\lambda}_K(x)$ for all $x\in\hat{\mathbb{C}}$.
\end{rem}

\section{On Lambda Equalities}\label{equality}

We prove Theorem \ref{equality-case}, establishing three equalities in terms of the lambda function. Two of these equalities are for $E$-compacta. The other one is for partially unshielded compacta.

Given an $E$-compactum  $K\subset\hat{\mathbb{C}}$ with complementary components $U_1,U_2,\ldots$, so that the diameters  $\delta(U_i)$ either form a finite sequence of an infinite one converging to zero. The Torhorst Inequality requires that $\sup\limits_i\lambda_{\partial U_i}(x)\le\lambda_K(x)$ for all $x\in\hat{\mathbb{C}}$ and for all $i\ge1$. Since $\lambda_K(x)=\tilde{\lambda}_K(x)=0$ for all $x\in K^o\cup\left(\bigcup\limits_iU_i\right)$, we only need to consider the points on $\partial K$, which may not equal $\bigcup\limits_i\partial U_i$.
%Note that a point in the difference $\partial K\setminus\bigcup\limits_i\partial U_i$ is called a {\em buried point} of $K$.

Lemma \ref{bridging_lemma} follows from \cite[Lemma 3.3]{LLY-2019} and is useful when we prove Lemma \ref{trivial-fiber}.
\begin{lemma}\label{bridging_lemma}
If $A\subset\hat{\bbC}$ is a closed topological annulus and $K\subset\hat{\bbC}$ a compactum then the following statements are equivalent: (1) $A\cap K$ has infinitely many components intersecting each of the two components of $\partial A$; (2) $A\setminus K$ has infinitely many components intersecting each of the two components of $\partial A$.
\end{lemma}

%\begin{lemma}\label{separation} Let $K\subset \bbC$ be a compact set and $U$ the region bounded by two parallel lines $L_{1}$ and $L_{2}$, such that $\partial U=L_1\cup L_2$. If \ $\overline{U}\setminus K$ has at least $m\ge2$ components intersecting both $L_{1}$ and $L_{2}$, then $\overline{U}\cap K$ has  at least $m-1$ components intersecting both $L_{1}$ and $L_{2}$. On the other hand, if \ $\overline{U}\cap K$ has at least $m\ge2$ components intersecting both $L_{1}$ and $L_{2}$, then $\overline{U}\setminus K$ has at least $m$ components intersecting both $L_{1}$ and $L_{2}$. \end{lemma}

%With the above lemma,  we find a relation between the fibers of $\overline{R_K}$ and those of $\overline{R_{\partial U_i}}$.

\begin{lemma}\label{trivial-fiber}
Given an $E$-compactum  $K\subset\hat{\mathbb{C}}$ with complementary components $U_1,U_2,\ldots$. If $\overline{R_K}[x]$ contains a point $y\ne x$ then $y\in\overline{R_{\partial U_i}}[x]$ for some $i$.
\end{lemma}
\begin{proof}
Let $\rho(x,y)$ be the spherical distance between $x$ and $y$. For each $n\ge2$ let $B_n(x)$ and $B_n(y)$ be the open disks of radius $2^{-n}\rho(x,y)$ that are centered at $x$ and $y$ respectively. Then $A_n=\hat{\mathbb{C}}\setminus\left(B_n(x)\cup B_n(y)\right)$ is a topological annulus. By \cite[Theorem 1.3]{LYY-2020}, the intersection $A_n\cap K$ has infinitely many components that intersect $\partial B_n(x)$ and $\partial B_n(y)$ both. By Lemma \ref{bridging_lemma}, the difference $A_n\setminus K$ has infinitely many components, say $\{P^n_j: j\ge1\}$, that intersect $\partial B_n(x)$ and $\partial B_n(y)$ both. Since the diameters of those $P^n_j$ are no less than $\rho(x,y)/2$ and since we assume $K$ to be an $E$-compactum, there is an integer $i(n)$ such that $U_{i(n)}$ contains infinitely many of those $P^n_j$. Here all those $P^n_j$ that are contained in $U_{i(n)}$ are each a component of $A_n\cap U_{i(n)}$.

Now, choose a subsequence $\{Q^n_k: k\ge1\}$ of $\{P^n_j: j\ge1\}$, with $Q_k^n\subset U_{i(n)}$, such that $\overline{Q^n_k}$ converges under Hausdorff distance to a continuum $M_n$. Then $M_n$ is a subset of $\partial U_{i(n)}$ and intersects  $\partial B_n(x)$ and $\partial B_n(y)$ both. Fixing any $a_n$ in $M_n\cap \partial B_n(x)$ and $b_n$ in $M_n\cap \partial B_n(y)$, we will have $(a_n,b_n)\in R_{\partial U_{i(n)}}$. Since $K$ is an $E$-compactum, there are infinitely many integers $n$ such that $i(n)$ takes the same value, say $i$. Therefore, we have two infinite sequences $\{c_n\}\subset\{a_n\}$ and $\{d_n\}\subset \{b_n\}$, with $c_n,d_n\in\partial U_i$, such that $(c_n,d_n)\in R_{\partial U_i}$ for all $n\ge2$. Since $\lim\limits_{n\rightarrow\infty}c_n=x$ and $\lim\limits_{n\rightarrow\infty}d_n=y$, we readily have $(x,y)\in\overline{R_{\partial U_i}}$, or equivalently $y\in\overline{R_{\partial U_i}}[x] $.
\end{proof}

Now we are well prepared to prove parts (i) and (ii) of Theorem \ref{equality-case}, whose results are respectively included in the next two propositions.

\begin{proposition}\label{equality-case-1}
If $K$ is an $E$-compactum such that $\tilde{\lambda}_K(x)=0$ for all $x\in\hat{\mathbb{C}}$ then $\lambda_K(x)$ vanishes everywhere.
\end{proposition}
\begin{proof}
As $\tilde{\lambda}_K(x)$ vanishes everywhere, all the relations $\overline{R_{\partial U_i}}$ are trivial, in the sense that the fibers $\overline{R_{\partial U_i}}[x]$ are each a singleton for all $i$ and all $x\in\partial U_i$. Combing this with the conclusion of Lemma \ref{trivial-fiber}, we can infer that the fiber $\overline{R_K}[x]=\{x\}$ for all $x\in K$. From this, we see that every atom of $K$ is a singleton and that $\lambda_K(x)=0$ for all $x$.
\end{proof}

\begin{proposition}\label{equality-case-2}
Given an $E$-compactum $K$. If $\partial U_i\cap\partial U_j=\emptyset$ for $i\ne j$ then $\lambda_K=\tilde{\lambda}_K$.
\end{proposition}
\begin{proof}
Let $\mathcal{D}_i$ denote the core decomposition of $\partial U_i$. Since we assume that $\partial U_i\cap\partial U_j=\emptyset$ for $i\ne j$, the collection
$\displaystyle \mathcal{D}_K^*:=\left(\bigcup\limits_i\mathcal{D}_i\right)\cup\left\{\{x\}: x\in K\setminus\left(\bigcup\limits_i\partial U_i\right)\right\}$
is a partition that divides $K$ into sub-continua. It suffices to show that $\Dc_K^*$ is the core decomposition of $K$.

Recall that $\Dc_K$ is the finest monotone decomposition such that every fiber of $\overline{R_K}$ is contained in a single element of $\Dc_K$. By Lemma \ref{key-lemma}, we know that $\Dc_K$ is refined by $\Dc_K^*$. On the other hand, since $K$ is an $E$-compactum and since $\partial U_i\cap\partial U_j=\emptyset$ for $i\ne j$, we can use Lemma \ref{trivial-fiber} to infer that every fiber of $\overline{R_K}$ is contained in a single element of $\Dc^*_K$. Therefore, we only need to verify that $\mathcal{D}^*_K$ is upper semi-continuous, which then indicates that $\Dc_K^*$ is a monotone decomposition hence is refined by $\Dc_K$.

In other words, we need to verify that the equivalence $\sim$ determined by the partition $\Dc_K^*$ is closed as a subset of $K\times K$. To this end, we consider an arbitrary sequence $\{(x_n,y_n): n\ge1\}$ in $K\times K$ with $\lim\limits_{n\rightarrow\infty}(x_n,y_n)=(x,y)$ such that $x_n\sim y_n$ for all $n\ge1$. There are two possibilities: either $x=y$ or $x\ne y$. In the first case, we have $(x,y)=(x,x)$, which is surely an element of $\sim$. In the second, the assumption that $K$ is an $E$-compactum implies that there is some $U_i$ such that $\{x_n,y_n\}\subset\partial U_i$ for infinitely many $n\ge1$. Consequently, the subset $\{x,y\}$ is contained in a single element of $\Dc_{i}$, which is a sub-collection of $\Dc_K^*$. That is to say, we have $x\sim y$. This ends our proof.
\end{proof}

The arguments in the above proof actually imply the following.
\begin{theo}\label{equal-cd}
Given an $E$-compactum $K$. If $\partial U_i\cap\partial U_j=\emptyset$ for $i\ne j$ then every atom of $K$ is either an atom of some $\partial U_i$ or a singleton $\{x\}$ with $x\in K\setminus\left(\bigcup_i\partial U_i\right)$.
\end{theo}

Now we go on to consider partially unshielded compacta and obtain Theorem \ref{equality-case}(iii).

\begin{deff}\label{part-unshielded}
Let $L\subset\hat{\mathbb{C}}$ be an unshielded compactum, which equals the boundary $\partial U$ of one of its complementary components $U$. A compactum $K$ formed by the union of $L$ with some complementary components of $L$ other than $U$ is called a {\bf partially unshielded compactum} determined by $L$.
\end{deff}

In order to find typical examples, one may set $L$ to be a polynomial Julia set, $U$ the unbounded Fatou component, and $K$ the union of $L$ and some bounded Fatou components. The next proposition discusses the relation between the atoms of any given compactum $L\subset\hat{\mathbb{C}}$ and those of a compactum $K$, where $K$  is  the union of $L$ with some (not all) components of $\hat{\bbC}\setminus L$.

\begin{proposition}\label{useful}
Given a planar compactum $L\subset \hat{\mathbb{C}}$ and a family $\{U_\alpha:\ \alpha\in I\}$  of  components of \ $\hat{\mathbb{C}}\!\setminus\!L$. If $\displaystyle K=L\cup\left(\bigcup_{\alpha\in I}U_\alpha\right)$ then  $\overline{R_K}$ is a subset of  $\{(z,z):\ z\in K\!\setminus\!L\}\cup\overline{R_L}$. Consequently, every atom of $K$ is either a singleton lying in $K\setminus L$ or a sub-continuum  of an atom of $L$.
\end{proposition}

\begin{proof}
Since $\displaystyle K=L\cup\left(\bigcup_{\alpha\in I}U_\alpha\right)$, every point $z\in (K\setminus L)$ lies in some $U_\alpha$. Thus the atom of $K$ containing $z$ is exactly the singleton $\{z\}$. From this it readily follows that  every atom $d^*$ of $K$ that intersects $L$ is a sub-continuum of $L$. So we have $\overline{R_K}=\{(z,z):\ z\in K\!\setminus\!L\}\cup\left(\overline{R_K}\cap L^2\right)$. Therefore, we only need to show that $\left(\overline{R_K}\cap L^2\right)\subset \overline{R_L}$.

Indeed, if on the contrary there were some $(x,y)\in \overline{R_K}\cap L^2$ not belonging to $\overline{R_L}$ then, for any small enough number $r>0$, the difference $L\setminus (B_r(x)\cup B_r(y))$ would have finitely many components intersecting $\partial B_r(x)$ and $\partial B_r(y)$ both. Let $A_r=\hat{\mathbb{C}}\setminus (B_r(x)\cup B_r(y))$. By  Lemma \ref{bridging_lemma},   $A_r\setminus L$  has at most finitely many components that intersect $\partial B_r(x)$ and $\partial B_r(y)$ both. As we assume that  $\displaystyle K=L\cup\left(\bigcup_{\alpha\in I}U_\alpha\right)$, every component of $A_r\setminus K$ is also a component of $A_r\setminus L$. Thus $A_r\setminus K$ has at most finitely many components that intersect both $\partial B_r(x)$ and $\partial B_r(y)$. In other words, we have $(x,y)\notin\overline{R_K}$. This is absurd since we assume that $(x,y)\in \overline{R_K}$.
\end{proof}

There are other basic facts concerning an unshielded compactum $L$  and a partially unshielded compactum $K$ determined by $L$. Firstly, every interior point of $K$ lies in some complementary component of $L$; secondly, every boundary point of $K$ lies in $L$. Thus we always have $\partial K=L$; moreover, every atom of $K$ that intersects the interior $K^o$ is necessarily a singleton. Therefore, in order to determine the atoms of $K$ we only need to consider those of $L$.

\begin{theo}\label{part-2}
Let $L\subset\hat{\mathbb{C}}$ be  an unshielded compactum. Let $K$ be a partially unshielded compactum determined by $L$. Then every atom of $L$ is also an atom of $K$ and we have $\Dc_K=\Dc_L\cup\{\{x\}: x\in K\setminus L\}$.  Consequently, $\tilde{\lambda}_K(x)=\lambda_K(x)$ for all $x\in\hat{\mathbb{C}}$.
\end{theo}

\begin{proof}
As $L$ is unshielded, there is a component $U$ of $\hat{\mathbb{C}}\setminus L$ with $L=\partial U$. By Lemma \ref{key-lemma},  every atom of $L$ lies in a single atom of $K$. By Lemma \ref{useful}, every atom of $K$ intersecting $L$ is contained in a single atom of $L$. Thus every atom of $L$ is also an atom of $K$. As any singleton $\{x\}$ with $x\in K^o= K\setminus L$ is an atom of $K$, we have $\Dc_K=\Dc_L\cup\{\{x\}: x\in K\setminus L\}$. This indicates the Lambda Equality $\tilde{\lambda}_K=\lambda_K$.
\end{proof}

\begin{rem}\label{why_partially_unshielded}
Theorem \ref{part-2} gives a result that is slightly stronger than  Theorem \ref{equality-case}(iii). In particular, for any full compactum $K$ we have $\mathcal{D}_{\partial K}\subset\mathcal{D}_K$. Therefore, a full compactum $K$ is a Peano compactum if and only if the boundary $\partial K$ is. In particular, if  $G\subset\hat{\mathbb{C}}$ is  a simply connected bounded domain then $\partial G$ is locally connected if and only if $K=\hat{\mathbb{C}}\setminus G$ is locally connected, or equivalently when $K$ is a Peano continuum. This basic fact has been well known, see for instance the items (iii) and (iv) of \cite[p.20, Theorem 2.1]{Pom92}. Now, it is extended to a quantitative version in Theorem \ref{part-2}. This extension applies to an arbitrary full continuum, that may or may not be locally connected.
\end{rem}

\section{The Gluing Lemma for Lambda Functions}\label{glue}

We will follow the philosophy of the well known gluing lemma for continuous maps.
See for instance \cite[p.69, Theorem (4.6)]{Armstrong} for the simple case and \cite[p.70, Theorem (4.8)]{Armstrong} for the general setting. Our aim is to prove Theorem \ref{gluing_lemma}, which deals with the lambda functions $\lambda_K,\lambda_L$ for planar compacta  $K\supset L$  such that $A=\overline{K\setminus L}$ intersects $L$ at finitely many points $x_1,\ldots,x_n$. In Theorem \ref{baby_M}, we further extend Theorem \ref{gluing_lemma} to the case that $A\cap L$ is a countably infinite set, under additional assumptions. Notice that when $A\cap L$ is an infinite set Theorem \ref{gluing_lemma} may not hold. See Examples \ref{cantor_combs}, \ref{brooms} and \ref{cup-fs}.

\begin{proof}[{\bf Proof for Theorem \ref{gluing_lemma}}]
For $1\le i\le n$, denote by $d_i^1$ the order-$1$ atom of $A$ that contains $x_i$. Similarly, denote by $e_i^1$  the atom of $L$ that contains $x_i$.
Let $K_1=A_1\cup L_1$, where $\displaystyle A_1=\bigcup_id_i^1$ and $\displaystyle L_1=\bigcup_ie_i^1$. Then $K_1$ has finitely many components. Let $\Ec_1$ be the collection of these components.

By \cite[Theorem 1.3]{LYY-2020},  a point $y\ne x$ lies in $\overline{R_K}[x]$
if and only if $K\setminus(B_r(x)\cup B_r(y))$ has infinitely many components that intersect both $\partial B_r(x)$ and $\partial B_r(y)$ for small enough  $r>0$. Because of this, we can directly check that $\overline{R_K}=\overline{R_A}\cup\overline{R_L}$. Here $\overline{R_K},\overline{R_A},\overline{R_L}$ are respectively  the closed Sch\"onflies relations on $K,A$ and $L$.  Let \[
\Dc_1=\left(\Dc_L\setminus\left\{e_1^1,\ldots,e_n^1\right\}\right)\cup
\left(\Dc_A\setminus\left\{d_1^1,\ldots,d_n^1\right\}\right)\cup
\Ec_1.\]
Then $\Dc_1$ is an upper semi-continuous decomposition of $K$ into subcontinua. Since $\Dc_1$ does not split the fibers of $\overline{R_K}$, it is refined by $\Dc_K$, the core decomposition of $K$ with Peano quotient. On the other hand, the equality $\overline{R_K}=\overline{R_A}\cup\overline{R_L}$ indicates that $\mathcal{D}_K$ does not split the fibers of $\overline{R_A}$ and those of $\overline{R_L}$. Thus each atom of $A$ lies in an atom of $K$; similarly, every atom of $L$ lies in an atom of $K$. Consequently, we have.
\begin{lemma}\label{gluing_atoms_a}
$\Dc_K=\Dc_1$. Thus $d\cap A$ (or $d\cap L$) either is empty or consists of finitely many atoms of $A$ (resp. $L$) for any atom $d$ of $K$.
\end{lemma}

Lemma \ref{gluing_atoms_a} ensures that $\displaystyle \lambda_K(x)=\max\left\{\lambda_A(x),\lambda_L(x)\right\}$ for all $x\notin K_1$.  That is to say, the equation $\lambda_K(x)=\max\left\{\lambda_{\overline{K\setminus L}}(x),\lambda_L(x)\right\}$ in Theorem \ref{gluing_lemma} holds for all points $x\notin K_1$, so that we only need to consider the points $x\in K_1$.

Notice that we have set $A=\overline{K\setminus L}$, $\displaystyle A_1=\bigcup_id_i^1$, and $\displaystyle L_1=\bigcup_ie_i^1$. We will need to verify that $\displaystyle \lambda_{K_1}(x)=\max\left\{\lambda_{A_1}(x),\lambda_{L_1}(x)\right\}$ for all $x\in K_1$, since for $x\in A_1$ and $y\in L_1$ we have
\[
\begin{array}{ccc}
\lambda_A(x)=\left\{\begin{array}{ll} 0& \{x\}\in\Dc_A\\
1+\lambda_{A_1}(x)& otherwise\end{array}\right. &\text{and}&
\lambda_L(y)=\left\{\begin{array}{ll} 0& \{y\}\in\Dc_L\\
1+\lambda_{L_1}(y)& otherwise.\end{array}\right.
\end{array}\]
To do that, we recall that  $\mathcal{D}_{A_1}$ consists of all the order-$2$ atoms of $A$ lying in $A_1$. Similarly, $\mathcal{D}_{L_1}$ consists of all the order-$2$ atoms of $L$ lying in $L_1$. Thus we may repeat the above procedure again, replacing $A$ and $L$ by $A_1$ and $L_1$. This then gives rise to two compacta $A_2\subset A_1$ and $L_2\subset L_1$ such that
$\displaystyle \lambda_{K_1}(x)=\max\left\{\lambda_{A_1}(x),\lambda_{L_1}(x)\right\}$ for all $x\notin K_2=A_2\cup L_2$.

We may carry out the same procedure indefinitely and obtain two decreasing sequences of compacta: (1) $A_1\supset A_2\supset\cdots$ and (2) $L_1\supset L_2\supset\cdots$. Setting $K_p=A_p\cup L_p$ for $p\ge1$, we have the following equations:
\begin{equation}
\displaystyle \lambda_{K_{p}}(x)=\left\{\begin{array}{ll}0& \{x\}\in\Dc_{K_p}\\
1+\lambda_{K_{p+1}}(x)& otherwise\end{array}\right. \quad (x\in K_{p+1}).
\end{equation}
\begin{equation}
\displaystyle \lambda_{A_{p}}(x)=\left\{\begin{array}{ll}0& \{x\}\in\Dc_{A_p}\\
1+\lambda_{A_{p+1}}(x)& otherwsie\end{array}\right. \quad(x\in A_{p+1})
\end{equation}
\begin{equation}
\displaystyle \lambda_{L_{p}}(x)=\left\{\begin{array}{ll}0& \{x\}\in\Dc_{L_p}\\
1+\lambda_{L_{p+1}}(x)& otherwise\end{array}\right.\quad (x\in L_{p+1})
\end{equation}
\begin{equation}
\lambda_{K_p}(x)=\max\left\{\lambda_{A_p}(x),\lambda_{L_p}(x)\right\}\quad (x\notin K_{p+1})
\end{equation}
There are two possibilities. In the first, we have $K_p=K_{p+1}$  for some $p\ge1$, indicating that $K_m=K_p$ for all $m\ge p$.
In such a case, we have $\lambda_{K_p}(x)=\max\left\{\lambda_{A_p}(x),\lambda_{L_p}(x)\right\}$ and hence $\lambda_{K}(x)=\max\left\{\lambda_{A}(x),\lambda_{L}(x)\right\}$.
In the second, we have $K_p\ne K_{p+1}$ for all $p\ge1$. This implies that $\lambda_K(x)=\max\left\{\lambda_A(x),\lambda_L(x)\right\}=\infty$ holds for all $x\in K_\infty=\bigcap_pK_p$ and that $\lambda_{K}(x)=p+\lambda_{K_p}(x)=p+\max\left\{\lambda_{A_p}(x),\lambda_{L_p}(x)\right\}=
\max\left\{\lambda_{A}(x),\lambda_{L}(x)\right\}$ holds for $p\ge1$ and $x\notin K_p\setminus K_{p+1}$. Here $\displaystyle K_1\setminus K_{\infty}=\bigcup_{p=1}^\infty(K_p\setminus K_{p+1})$. This completes our proof.
\end{proof}

Lemma \ref{gluing_atoms_a} and Theorem \ref{gluing_lemma} are useful, when we study $\lambda_K$ for certain choices of planar compacta $K$. For instance, we may choose $K$ to be the Julia set of a renormalizable polynomial $f(z)=z^2+c$ and $L$ the small Julia set. For the sake of convenience, we further assume that the only critical point of $f$ is recurrent and that there is no irrationally neutral cycle. Then it is possible to choose a decreasing sequence of Jordan domains $\{U_n\}$, with $\overline{U_{n+1}}\subset U_n$ and $\displaystyle L=\bigcap_{n=1}^\infty U_n$, such that every $K\cap \partial U_n$ consists of finitely many points that are periodic or pre-periodic. See for instance \cite[section 2.2]{Jiang00}.
%Let $\mathcal{D}_n$ denote the core decomposition of $L_n=K\cap\overline{U_n}$ with Peano quotient.
For any $n\ge1$ we can use \cite[Theorems 2 and 3]{Kiwi04}  to infer that every singleton $\{x\}$ with $x\in (K\cap\partial U_n)$ is an atom of $K$ hence is also an atom of $L_n=K\cap\overline{U_n}$. Combining these with Lemma \ref{gluing_atoms_a} and Theorem \ref{gluing_lemma},  we further see that $\mathcal{D}_{L_{n+1}}\subset\mathcal{D}_{L_n}\subset\mathcal{D}_K$ for all $n\ge1$. However, we are not sure whether $\mathcal{D}_L\subset\mathcal{D}_K$. Similarly, it is not clear whether $\lambda_K(x)=\lambda_L(x)$ holds for $x\in L$. Therefore, we propose the following.

\begin{que}\label{small_julia}
Let $K=L_0\supset L_1\supset L_2\supset\cdots$ a decreasing sequence of planar compacta such that  $L_n\cap\overline{K\setminus L_n}$ is a finite set for all $n\ge1$. % and that (2) for any $n\ge1$ and any $x\in L_n\cap\overline{K\setminus L_n}$ the singleton $\{x\}$ is an atom of $K$.
Setting $L=\bigcap_{n\ge1}L_n$. Find conditions so that (1) $\mathcal{D}_L\subset\mathcal{D}_K$ or  (2) $\lambda_K(x)=\lambda_L(x)$ holds for all $x\in L$.
\end{que}

As a response to Question \ref{small_julia}, we turn to study the lambda functions of two planar compact $K\supset L$
such that $K\setminus L$ is contained in the union of at most countably many continua  $P_n\subset K$ that satisfy the following properties:
\begin{itemize}
\item[(P1)] every $P_n$ intersects $L$ at a single point $x_n$, and
\item[(P2)] for any constant $C>0$ at most finitely many $P_n$ are of diameter greater than $C$.
\item[(P3)] $P_n\cap P_m=\emptyset$ for $n\ne m$.
\end{itemize}
Here $P_n\setminus\{x_n\}$ might be disconnected for some of the integers $n\ge1$. Notice that there is a special situation, when $K$ is the Mandelbrot set $\M$. Then, in order that the above properties (P1)-(P3) be satisfied, we may choose $L$ to be  the closure of a hyperbolic component or   a {\bf Baby Mandelbrot set}.

As an extension of Theorem \ref{gluing_lemma}, we will obtain the following.

\begin{theo}\label{baby_M}
Given two planar compacta $K\supset L$ that satisfy (P1) to (P3),  we have \begin{equation}\label{baby}
\lambda_K(x)=\left\{\begin{array}{lll} \lambda_{P_n}(x)&x\in P_n\setminus\{x_n\}\ {for\ some}\ n&({case}\ 1)\\ \lambda_L(x)& x\in L\setminus\{x_n: n\in\mathbb{N}\}&({case}\ 2)\\ \max\left\{\lambda_L(x_n),\lambda_{P_n}(x_n)\right\}& x=x_n\ {for\ some}\ x_n&({case}\ 3)\\
0 &{otherwise}&(case\ 4)\end{array}\right. \end{equation}
\end{theo}

We just need to consider the above equation for points $x\in K$.
%In other words, we only need to verify the following equation:
%\begin{equation}\label{inductive_0} \lambda_K(x)=\left\{\begin{array}{lll} \lambda_{P_n}(x)&x\in P_n\setminus\{x_n\}\ {for\ some}\ n&({case}\ 1)\\ \lambda_L(x)& x\in L\setminus\{x_n: n\in\mathbb{N}\}&({case}\ 2)\\ \max\left\{\lambda_L(x_n),\lambda_{P_n}(x_n)\right\}& x=x_n\ {for\ some}\ x_n&({case}\ 3)\end{array}\right. \end{equation}
To do that, we may
%assume with no loss of generality that the diameter of $P_n$ is no less than that of $P_{n+1}$ for all $n\ge1$.  Then, we
define an equivalence $\sim$ on $\mathbb{N}$ so that $m\sim n$ if and only if $x_m,x_n$ are  contained in the same atom of $L$. Let $\{I_j:j\}$ be the equivalence classes of $\sim$.

Denote by $d_n$ the atom of $P_n$ that contains $x_n$, and by $e_j$ the atom of $L$ that contains all $x_n$ with $n\in I_j$. Moreover, set $e'_j=e_j\cup\left(\bigcup_{n\in I_j}d_n\right)$ for every $j$.

Then $\{e'_j: j\}$ is a collection of at most countably many continua that are pairwise disjoint.
Now we consider the following upper semi-continuous decomposition of $K$:
\begin{equation}\label{baby_M_partition}
\Dc_1=\left(\Dc_L\setminus\left\{e_j: j\right\}\right)\cup
\left(\bigcup_n\Dc_{P_n}\setminus\{d_n\}\right)\cup
\{e'_j: j\}.
\end{equation}
All its elements are sub-continua of $K$ that do not split the fibers of $\overline{R_K}$. So it is refined by $\Dc_K$. On the other hand, by \cite[Theorem 1.3]{LYY-2020}, we also have $\overline{R_K}=\overline{R_L}\cup\left(\bigcup_{n\ge1}\overline{R_{P_n}}\right)$. Thus $\mathcal{D}_K$ does not split the fibers of $\overline{R_L}$ and those of $\overline{R_{P_n}}$ for all $n\ge1$. Therefore, every atom of $L$ lies in an atom of $K$, so does  every atom of $P_n$. Consequently, we have.
\begin{lemma}\label{gluing_atoms_b}
$\Dc_K=\Dc_1$. Therefore, for any atom $d$ of $K$  the intersection $d\cap L$ (or $d\cap P_n$ for any $n\ge1$) is either empty or a single atom of $L$ (respectively, $P_n$).
\end{lemma}

\begin{proof}[{\bf Proof for Theorem \ref{baby_M}}]
Clearly,  $\lambda_K(x)=\lambda_L(x)$ for all $x$ in $L\setminus\left(\bigcup_je'_j\right)$. Similarly,  $\lambda_K(x)=\lambda_{P_n}(x)$ for all $x$ in $P_n\setminus d_n$. Moreover, $\lambda_K(x_n)=\lambda_L(x_n)=\lambda_{P_n}(x_n)=0$ for all $x_n$ such that $\{x_n\}$ is an atom of $L$ and also an atom of $P_n$. Therefore, we just need to consider those $e_j'$  that are non-degenerate.

Let $\mathcal{N}_1$ be the collection of all the integers $j$ such that $e_j'$ is not a singleton.
Then $e_{n_1}$ is a subcontinuum of $e_{n_1}'$ for any $n_1\in\mathcal{N}_1$, such that $e_{n_1}'\setminus e_{n_1}$ is covered by all those $d_n$ with $n\in I_{n_1}$. Thus the properties (P1) - (P3) are satisfied, if $K$ and $L$ are respectively replaced by $e_{n_1}'$ and $e_{n_1}$. It is then routine to check the following:
\begin{equation}\label{inductive_1}
\lambda_K(x)=\left\{\begin{array}{ll}
\lambda_{P_n}(x)&x\in P_n\setminus\left(\bigcup_{n_1}e_{n_1}'\right)\\
\lambda_{L}(x)& x\in L\setminus\left(\bigcup_{n_1}e_{n_1}'\right)\\
1+\lambda_{e_{n_1}'}(x)& x\in e_{n_1}'\ \text{for\ some}\ n_1\in\mathcal{N}_1
%\\ 0& x\notin K
\end{array}\right.
\end{equation}

Every atom of $e_{n_1}'$ falls into exactly one of the following possibilities: (1) an order-two atom of $P_n$ for some $n\in I_{n_1}$ that is disjoint from $\{x_n\}$, (2) an order-two atom of  $L$ that is disjoint from  $\{x_n: n\in I_{n_1}\}$,  (3) a singleton $\{x_n\}$ for some $n\in I_{n_1}$, which is an order-two atom of $L$ and is also an order-two atom of  $P_n$, (4) a non-singleton continuum that consists of the order-two atom of $L$ containing some $x_n$, with $n\in I_{n_1}$, and the order-two atom of $P_n$ containing $x_n$.

An atom falling in the first three possibilities  is called an atom of {\bf pure type}.
%An atom falling in the last possibility  is called an atom of {\bf exceptional type}.
We can check that $e_{n_1}'$ has at most countably many atoms that is not of pure type. Such an atom is generally denoted as $e'_{n_1n_2}$. Similarly we can define continua $e'_{n_1n_2\ldots n_p}$ for $p>2$. On the one hand, such a continuum  is an order-$p$ atom of $K$; on the other, it is also an atom of $e'_{n_1n_2,\ldots n_{p-1}}$ that is not of pure type.
By the same arguments, that have been used in obtaining Equation (\ref{inductive_1}), we can infer the following equation:
\begin{equation}\label{inductive_2}
\lambda_K(x)=\left\{\begin{array}{ll}
\lambda_{P_n}(x)&x\in P_n\setminus\left(\bigcup_{n_1,n_2}e_{n_1n_2}'\right)\\
\lambda_{L}(x)& x\in L\setminus\left(\bigcup_{n_1,n_2}e_{n_1n_2}'\right) \\
2+\lambda_{e_{n_1n_2}'}(x)& x\in  e_{n_1n_2}'\ \text{for\ some}\ n_1,n_2
\end{array}\right.
\end{equation}
This equation may be extended to order-$p$ atoms $e'_{n_1n_2\ldots n_p}$ with $p\ge2$ in the following way.
\begin{equation}\label{inductive_p}
\lambda_K(x)=\left\{\begin{array}{ll}
\lambda_{P_n}(x)&x\in P_n\setminus\left(\bigcup_{n_1,\ldots,n_p}e_{n_1\cdots n_p}'\right)\\
\lambda_{L}(x)& x\in L\setminus\left(\bigcup_{n_1,\ldots,n_p}e_{n_1\cdots n_p}'\right) \\
p+\lambda_{e_{n_1\cdots n_p}'}(x)& x\in  e_{n_1\cdots n_p}'\ \text{for\ some}\ n_1,\ldots,n_p
\end{array}\right.
\end{equation}
Notice that Theorem \ref{baby_M} holds for every $x\in K$ lying in an atom  of  $e'_{n_1n_2\ldots n_p}$ that is of pure type. Such a point $x$ does not lie in $e'_{n_1n_2\ldots n_pn_{p+1}}$ for any choice of $n_{p+1}$ and hence falls into exactly one of the following possibilities:
\begin{itemize}
\item that $x\in P_n\setminus\{x_n\}$ for some $n\ge1$ and $\lambda_K(x)=\lambda_{P_n}(x)\ge p$.
\item that $x\in L\setminus\{x_n: n\ge1\}$ and $\lambda_K(x)=\lambda_L(x)\ge p$.
\item that $x=x_n$ for some $n\ge1$ and $\lambda_K(x)=\max\left\{\lambda_L(x),\lambda_{P_n}(x)\right\}=p$.
\end{itemize}
Every other point $x\in K$ necessarily lies in $e'_{n_1n_2\ldots n_p}$ for infinitely many $p$. The continua $e'_{n_1n_2\ldots n_p}$ decrease to a continuum $M_x$. There are three possibilities, either $x\in L\setminus\{x_n\}$, or  $x\in P_n\setminus\{x_n\}$ for some $n\ge1$, or $x=x_n$ for some $n\ge1$.  In the first case, we have $\lambda_K(x)=\lambda_L(x)=\infty$; in the second, we have  $\lambda_K(x)=\lambda_{P_n}(x)=\infty$; in the third, we have $\lambda_K(x)=\max\left\{\lambda_L(x),\lambda_{P_n}(x)\right\}=\infty$. This completes our proof.
\end{proof}

\begin{rem}
%We find a possible answer to Question \ref{small_julia}.
Let $K=L_0\supset L_1\supset L_2\supset\cdots$ be given as in Question \ref{small_julia}. Also let  $L=\bigcap_{n\ge1}L_n$.  Then $L_n\cap\overline{K\setminus L_n}$ is a finite set for all $n\ge1$.
Assume in addition that (1) every singleton $\{x_n\}$ is an atom of $K$ and (2) $K$ and $L$  satisfy the requirements in Theorem \ref{baby_M}. By Lemma \ref{gluing_atoms_a}, we see that $\mathcal{D}_{L_{n+1}}\subset\mathcal{D}_{L_n}\subset\mathcal{D}_K$ for all $n\ge1$; thus from Theorem \ref{gluing_lemma} we can infer that $\lambda_K(x)=\lambda_{L_n}(x)$  for all $x\in L_n$. Moreover, by Lemma \ref{gluing_atoms_b} we have $\mathcal{D}_L\subset\mathcal{D}_K$. Therefore, by Theorem \ref{baby_M} we further infer that $\lambda_K(x)=\lambda_L(x)$ holds for all $x\in L$.
\end{rem}

\section{Examples}\label{examples}

We shall construct examples.
%trying to illustrate that assumptions in some of our theorems cannot be removed or  further weakened.
%\subsection{When Gluing Lemma does not work}\label{gluing_negative}
The beginning two  provide choices of compacta $A,B\subset\hat{\bbC}$ such that $\lambda_{A\cup B}(x)\ne\max\left\{\lambda_A(x),\lambda_B(x)\right\}$ for some $x$, although $\lambda_A(x)=\lambda_B(x)$ for all $x\in A\cap B$. In the first,  $A\cap B$ is an uncountable set; in the second,  $A\cap B$ is a countably infinite set. Therefore, the conditions of Theorem \ref{gluing_lemma} are not satisfied.

\begin{exam}\label{cantor_combs}
Let $A=%\{t: 0\le t\le1\}\cup
\{t+s{\bf i}: t\in\Kc, 0\le s\le1\}$, where $\Kc$ is the Cantor ternary set. Let
$B=%\{1+s{\bf i}: 1\le s\le 2\}\cup
\{t+(1+s){\bf i}: 0\le t\le 1, s\in\Kc\}$. Let $A_1=A\cup B$ and $B_1=(A+1+{\bf i})\cup(B+1-{\bf i})$.
See Figure \ref{not_glued} for a simplified depiction of $A, B, A_1, B_1$.
\begin{figure}[ht]
\vspace{-0.05cm}
\begin{tabular}{ccccc}
\begin{tikzpicture}[x=1cm,y=1cm,scale=0.5]
\foreach \i in {0,...,3}
{
    \draw[gray,very thick] (3*\i/27,0) -- (3*\i/27,3);
    \draw[gray,very thick] (6/9+3*\i/27,0) -- (6/9+3*\i/27,3);
     \draw[gray,very thick] (2+3*\i/27,0) -- (2+3*\i/27,3);
    \draw[gray,very thick] (2+6/9+3*\i/27,0) -- (2+6/9+3*\i/27,3);
}
\end{tikzpicture} \hspace{0.25cm}
&
\hspace{0.25cm}
\begin{tikzpicture}[x=1cm,y=1cm,scale=0.5]
\foreach \i in {0,...,3}
{
    \draw[gray,very thick] (0,3*\i/27+3) -- (3,3*\i/27+3);
    \draw[gray,very thick] (0,6/9+3*\i/27+3) -- (3,6/9+3*\i/27+3);
     \draw[gray,very thick] (0,2+3*\i/27+3) -- (3,2+3*\i/27+3);
    \draw[gray,very thick] (0,2+6/9+3*\i/27+3) -- (3,2+6/9+3*\i/27+3);
}
\draw[gray,dashed] (0,0) -- (3,0)-- (3,3)-- (0,3)-- (0,0);
\end{tikzpicture} \hspace{0.25cm}

&
\hspace{0.25cm}
\begin{tikzpicture}[x=1cm,y=1cm,scale=0.5]

\foreach \i in {0,...,3}
{
    \draw[gray,very thick] (3*\i/27,0) -- (3*\i/27,3);
    \draw[gray,very thick] (6/9+3*\i/27,0) -- (6/9+3*\i/27,3);
     \draw[gray,very thick] (2+3*\i/27,0) -- (2+3*\i/27,3);
    \draw[gray,very thick] (2+6/9+3*\i/27,0) -- (2+6/9+3*\i/27,3);
}
\foreach \i in {0,...,3}
{
    \draw[gray,very thick] (0,3*\i/27+3) -- (3,3*\i/27+3);
    \draw[gray,very thick] (0,6/9+3*\i/27+3) -- (3,6/9+3*\i/27+3);
     \draw[gray,very thick] (0,2+3*\i/27+3) -- (3,2+3*\i/27+3);
    \draw[gray,very thick] (0,2+6/9+3*\i/27+3) -- (3,2+6/9+3*\i/27+3);
}
\end{tikzpicture} \hspace{0.25cm}

&
\hspace{0.25cm}
\begin{tikzpicture}[x=1cm,y=1cm,scale=0.5]
\foreach \i in {0,...,3}
{
    \draw[gray,very thick] (3+3*\i/27,3) -- (3+3*\i/27,6);
    \draw[gray,very thick] (3+6/9+3*\i/27,3) -- (3+6/9+3*\i/27,6);
     \draw[gray,very thick] (3+2+3*\i/27,3) -- (3+2+3*\i/27,6);
    \draw[gray,very thick] (3+2+6/9+3*\i/27,3) -- (3+2+6/9+3*\i/27,6);
}
\foreach \i in {0,...,3}
{
    \draw[gray,very thick] (3,3*\i/27) -- (3+3,3*\i/27);
    \draw[gray,very thick] (3,6/9+3*\i/27) -- (3+3,6/9+3*\i/27);
     \draw[gray,very thick] (3,2+3*\i/27) -- (3+3,2+3*\i/27);
    \draw[gray,very thick] (3,2+6/9+3*\i/27) -- (3+3,2+6/9+3*\i/27);
}
\end{tikzpicture} \hspace{0.25cm}

&
\hspace{0.25cm}
\begin{tikzpicture}[x=1cm,y=1cm,scale=0.5]

\foreach \i in {0,...,3}
{
    \draw[gray,very thick] (3*\i/27,0) -- (3*\i/27,3);
    \draw[gray,very thick] (6/9+3*\i/27,0) -- (6/9+3*\i/27,3);
     \draw[gray,very thick] (2+3*\i/27,0) -- (2+3*\i/27,3);
    \draw[gray,very thick] (2+6/9+3*\i/27,0) -- (2+6/9+3*\i/27,3);
}
\foreach \i in {0,...,3}
{
    \draw[gray,very thick] (0,3*\i/27+3) -- (3,3*\i/27+3);
    \draw[gray,very thick] (0,6/9+3*\i/27+3) -- (3,6/9+3*\i/27+3);
     \draw[gray,very thick] (0,2+3*\i/27+3) -- (3,2+3*\i/27+3);
    \draw[gray,very thick] (0,2+6/9+3*\i/27+3) -- (3,2+6/9+3*\i/27+3);
}
\foreach \i in {0,...,3}
{
    \draw[gray,very thick] (3+3*\i/27,3) -- (3+3*\i/27,6);
    \draw[gray,very thick] (3+6/9+3*\i/27,3) -- (3+6/9+3*\i/27,6);
     \draw[gray,very thick] (3+2+3*\i/27,3) -- (3+2+3*\i/27,6);
    \draw[gray,very thick] (3+2+6/9+3*\i/27,3) -- (3+2+6/9+3*\i/27,6);
}
\foreach \i in {0,...,3}
{
    \draw[gray,very thick] (3,3*\i/27) -- (3+3,3*\i/27);
    \draw[gray,very thick] (3,6/9+3*\i/27) -- (3+3,6/9+3*\i/27);
     \draw[gray,very thick] (3,2+3*\i/27) -- (3+3,2+3*\i/27);
    \draw[gray,very thick] (3,2+6/9+3*\i/27) -- (3+3,2+6/9+3*\i/27);
}
\end{tikzpicture}
\\ $A$& $B$& $A_1=A\cup B$& $B_1$& $A_1\cup B_1$\end{tabular}
\caption{The two compacta $A, B$ and their union.}\label{not_glued}
\end{figure}
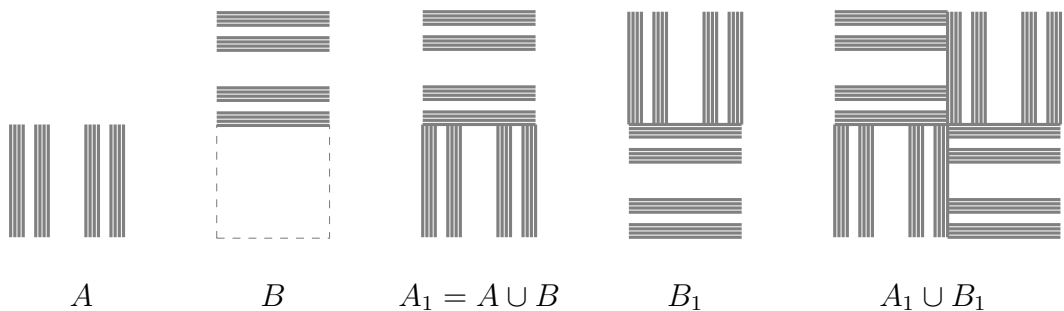
Then
$\lambda_A(x)=1$ for all $x\in A$ and vanishes otherwise; similarly,  $\lambda_B(x)=1$ for all $x\in B$ and vanishes otherwise.
%That is to say, we have $\lambda_{A}(x)=\left\{\begin{array}{ll}1&x\in A\\  0& {otherwise}\end{array}\right.$ and $ \lambda_{B}(x)=\left\{\begin{array}{ll}1&x\in B\\  0& {otherwise}.\end{array}\right.$
However, both $A\cap B$ and $A_1\cap B_1$ are uncountable, thus the conditions in Theorem \ref{gluing_lemma} are not satisfied. Moreover, we have
\[\lambda_{A_1}(x)=\lambda_{A\cup B}(x)=\left\{\begin{array}{ll}2&x\in A\\ 1& B\setminus A\\ 0& {otherwise}\end{array}\right.\quad {and}\quad
\lambda_{A_1\cup B_1}(x)=\left\{\begin{array}{ll}\infty & x\in (A_1\cup B_1)\\  0& {otherwise}.\end{array}\right.\]

\end{exam}

\begin{exam}\label{brooms}
Set $A=\bigcup\limits_{n\ge0}A_n$. Here $A_0=\{s{\bf i}: 0\le s\le1\}$ and $A_1$ is the continuum that consists of the line $\displaystyle\left\{1+t{\bf i}: 0\le t\le1\right\}$ and all those lines connecting $1+{\bf i}$ to $\displaystyle\frac{k}{k+1}$ for $k\ge1$; moreover, for $n\ge2$, $A_n=\displaystyle \left\{2^{-n+1}t+s{\bf i}: t+s{\bf i}\in A_1\right\}$.  See  Figure \ref{broom_comb}.
\begin{figure}[ht]
\begin{center}
\begin{tabular}{cc}
\begin{tikzpicture}[x=1.618cm,y=1cm,scale=1]
\foreach \i in {0,...,3}
{
    \draw[gray,very thick] (0,3+3*\i/27) -- (3,3+3*\i/27);
    \draw[gray,very thick] (0,3+6/9+3*\i/27) -- (3,3+6/9+3*\i/27);
     \draw[gray,very thick] (0,3+2+3*\i/27) -- (3,3+2+3*\i/27);
    \draw[gray,very thick] (0,3+2+6/9+3*\i/27) -- (3,3+2+6/9+3*\i/27);
}

\draw[gray,very thick] (0,0) --(0,3);
\draw[gray,very thick] (3,0) --(3,3);
\draw[gray,very thick] (3/2,0) --(3/2,3);
\draw[gray,very thick] (3/4,0) --(3/4,3);
\draw[gray,very thick] (3/8,0) --(3/8,3);
\draw[gray,very thick] (3/16,0) --(3/16,3);
\draw[gray,very thick] (3/32,0) --(3/32,3);
\draw[gray,very thick] (3/64,0) --(3/64,3);
\draw[gray,very thick] (3/128,0) --(3/128,3);
\draw[gray,very thick] (3/256,0) --(3/256,3);
\draw[gray,very thick] (3/512,0) --(3/512,3);

\foreach \i in {2,...,7}
{
    \draw[gray,very thick] (3,3) -- (3-3/\i,0);
}
\node at (2.8,0.15) {$\ldots$};

\foreach \i in {2,...,7}
{
    \draw[gray,very thick] (3/2,3) -- (3/2-1.5/\i,0);
}

%\foreach \i in {2,...,7} {     \draw[gray,very thick] (3/4,3) -- (3/4-0.75/\i,0); }
 \node at (9/16,1.75) {$\vdots$};
  \node at (9/16,1.25) {$\vdots$};
\node at (-0.1,0.2){$0$}; \node at (3.1,0.2){$1$};
\node at (0,0){$\cdot$}; \node at (3,0){$\cdot$};
\node at (3,3){$\cdot$}; \node at (3,6){$\cdot$};
\node at (3.35,3.0){$1\!+\!{\bf i}$}; \node at (3.35,6.0){$1\!+\!2{\bf i}$};
\end{tikzpicture}
&
%\hspace{0.25cm}
\begin{tikzpicture}[x=1.618cm,y=1cm,scale=1]
\foreach \i in {0,...,3}
{
    \draw[gray,dashed] (0,3+3*\i/27) -- (3,3+3*\i/27);
    \draw[gray,dashed] (0,3+6/9+3*\i/27) -- (3,3+6/9+3*\i/27);
     \draw[gray,dashed] (0,3+2+3*\i/27) -- (3,3+2+3*\i/27);
    \draw[gray,dashed] (0,3+2+6/9+3*\i/27) -- (3,3+2+6/9+3*\i/27);
}

\draw[gray,very thick] (0,3) --(3,3);

\draw[gray,very thick] (0,0) --(0,3);
\draw[gray,very thick] (3,0) --(3,3);
\draw[gray,very thick] (3/2,0) --(3/2,3);
\draw[gray,very thick] (3/4,0) --(3/4,3);
\draw[gray,very thick] (3/8,0) --(3/8,3);
\draw[gray,very thick] (3/16,0) --(3/16,3);
\draw[gray,very thick] (3/32,0) --(3/32,3);
\draw[gray,very thick] (3/64,0) --(3/64,3);
\draw[gray,very thick] (3/128,0) --(3/128,3);
\draw[gray,very thick] (3/256,0) --(3/256,3);
\draw[gray,very thick] (3/512,0) --(3/512,3);

\node at (-0.1,0.2){$0$}; \node at (3.1,0.2){$1$}; \node at (3.35,3.0){$1\!+\!{\bf i}$};
\end{tikzpicture}\\ $A\cup B$ & $d$
\end{tabular}
\end{center}
\vskip -0.75cm
\caption{The compactum $A\cup B$ and the  atom $d$ of $A\cup B$.}\label{broom_comb}
\end{figure}
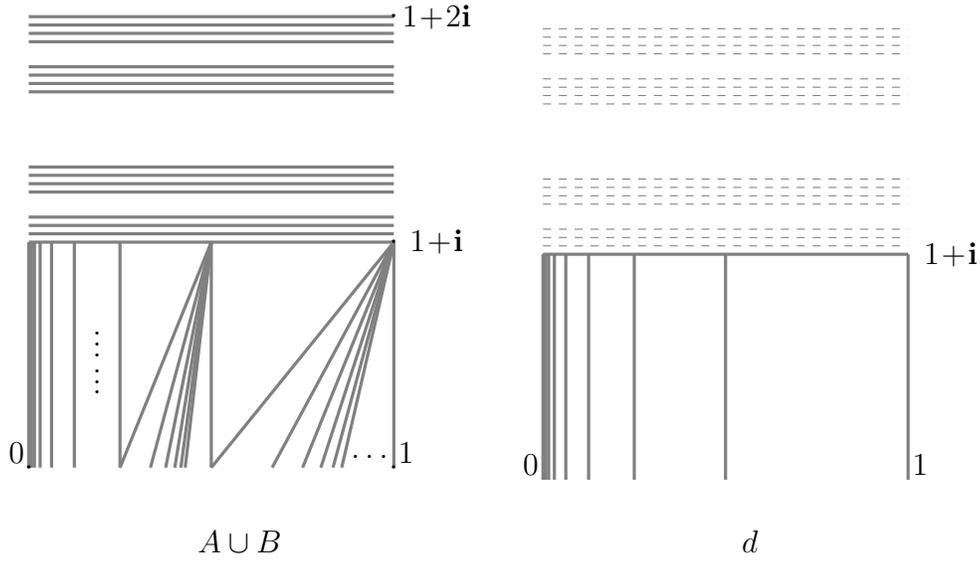
Further setting $B$ as in Example \ref{cantor_combs}, we have $A\cap B=\{ {\bf i}\}\cup\left\{2^{-n}+{\bf i}: n\ge0\right\}$.
If $x\in A\cap B$ then $\lambda_A(x)=\lambda_B(x)=1$. Let
$\displaystyle L_1=\left\{t+s{\bf i}:  0\le s\le 1, t=0\ {or}\ 2^{-n}\ {for\ some}\ n\ge0\right\}.$ Then $d=L_1 \cup\ \{t+{\bf i}: 0\le t\le1\}$ is an atom of $A\cup B$ and is not locally connected at any $x\in A_0$. Moreover, we have
\[\lambda_{A}(x)=\left\{\begin{array}{ll}1&x\in L_1 \\  0& {otherwise}\end{array}\right.\quad {and}\quad
\lambda_{A\cup B}(x)=\left\{\begin{array}{ll}2& x\in A_0\\ 1&x\in (B\cup d)\setminus A_0\\  0& {otherwise}.\end{array}\right.\]
\end{exam}

% \subsection{The Lambda Equality may fail for an $E$-compactum}\label{E-bad}

The next two examples are about $E$-continua $K\subset\hat{\mathbb{C}}$ such that the lambda equality given in (i) or (ii) of Theorem \ref{equality-case} does not hold.
%In the first one, we have $\lambda_K(x)-\tilde{\lambda}_K(x)=1$ for $x$ lying on a segment. In the second one,  we have $\lambda_K(x)=\infty$ and $\tilde{\lambda}_K(x)=1$ for all $x$ lying on a sub-continuum of $K$. Of course, these continua satisfy none of the assumptions in part (i), (ii), (iii) of Theorem \ref{equality-case}.

\begin{exam}\label{E-compactum}
Let $X$ denote the square $[1,2]\times[0,{\mathbf i}]\subset\hat{\mathbb{C}}$. Let $Y$ be an embedding of $[0,\infty)$ whose closure $\overline{Y}$ equals the union of $Y$ with $\partial X$. See the left part of Figure \ref{negative} for a simplified representation of \ $\overline{Y}$, which is depicted as \tb{blue}.
\begin{figure}[ht]
\vskip -0.25cm
\begin{tabular}{ll}
\begin{tikzpicture}[x=0.2cm,y=0.2cm,scale=0.55]
\draw[gray,thick] (64,0) -- (64,32) -- (0,32) -- (0,0) -- (64,0);

\foreach \j in {0,1}
{ \draw[gray, thick] (32,\j*16) -- (32,16+\j*16) -- (16,16+\j*16) -- (16,\j*16) --(32,\j*16);
}
\foreach \j in {0,...,3}
{
    \draw[gray, thick] (16,\j*8) -- (16,8+\j*8) -- (8,8+\j*8) -- (8,\j*8) --(16,\j*8);
}
\foreach \j in {0,...,7}
{
    \draw[gray, thick] (8,\j*4) -- (8,4+\j*4) -- (4,4+\j*4) -- (4,\j*4) --(8,\j*4);
}
\foreach \j in {0,...,15}
{
    \draw[gray, thick] (4,\j*2) -- (4,2+\j*2) -- (2,2+\j*2) -- (2,\j*2) --(4,\j*2);
}
\foreach \j in {0,...,31}
{
    \draw[gray, thick] (2,\j*1) -- (2,1+\j*1) -- (1,1+\j*1) -- (1,\j*1) --(2,\j*1);
}
\foreach \j in {0,...,63}
{
    \draw[gray, thick] (1,\j*1/2) -- (1,1/2+\j*1/2) -- (1/2,1/2+\j*1/2) -- (1/2,\j*1/2) --(1,\j*1/2);
}

\foreach \j in {0,...,127}
{
    \draw[gray, thick] (1/2,\j*1/4) -- (1/2,1/4+\j*1/4) -- (1/4,1/4+\j*1/4) -- (1/4,\j*1/4) --(1/2,\j*1/4);
}
\foreach \j in {0,...,255}
{
    \draw[gray, thick] (1/4,\j*1/8) -- (1/4,1/8+\j*1/8) -- (1/8,1/8+\j*1/8) -- (1/8,\j*1/8) --(1/4,\j*1/8);
}

\foreach \k in {0,1}
{
\draw[gray, dashed, thick] (16+1,16*\k+1) -- (32-1,16*\k+1) -- (32-1, 16*\k+16-1) -- (16+1/2,16*\k+16-1);
}

\foreach \k in {0,1,2,3}
{
\draw[gray, dashed, thick] (8+1/2,8*\k+1/2) -- (16-1/2,8*\k+1/2) -- (16-1/2, 8*\k+8-1/2) -- (8+1/2,8*\k+8-1/2);
}

\foreach \i in {0,1}
{\foreach \j in {2,...,6}
{
    \draw[gray, thick] (16+\j,\j+16*\i) -- (32-\j,\j+16*\i) -- (32-\j,16-\j+16*\i) -- (16+\j-1,16-\j+16*\i)--(16+\j-1,\j-1+16*\i);
}
}

\foreach \i in {0,...,3}
{\foreach \j in {2,...,6}
{
    \draw[gray] (8+\j/2,\j/2+8*\i) -- (16-\j/2,\j/2+8*\i) -- (16-\j/2,8-\j/2+8*\i) --
    (8+\j/2-1/2,8-\j/2+8*\i)--(8+\j/2-1/2,\j/2-1/2+8*\i);
}
}
\foreach \j in {0,...,7}
{
 \fill(5.5,2+\j*4)circle(1pt);
 \fill(6,2+\j*4)circle(1pt);
 \fill(6.5,2+\j*4)circle(1pt);
}
\node at (0.25,-1.5){$0$};
\node at (32,-1.5){$1$};
\node at (64,-1.5){$2$};
\node at (0.25,33.5){${\mathbf i}$};

\draw[blue,thick] (64,0) -- (64,32) -- (32,32) -- (32,0) -- (64,0);

\foreach \j in {2,...,6}
{    \draw[blue, thick] (32+2*\j,2*\j) -- (64-2*\j,2*\j) -- (64-2*\j,32-2*\j) -- (32+2*\j-2,32-2*\j) --(32+2*\j-2,2*\j-2);
}

\draw[blue, dashed, thick] (32+2,2) -- (64-2,2) -- (64-2,32-2) -- (32+1,32-2);

\end{tikzpicture}
&
\begin{tikzpicture}[x=0.2cm,y=0.2cm,scale=0.55]
\draw[gray,thick] (64,0) -- (64,32) -- (0,32) -- (0,0) -- (64,0);
\foreach \j in {0,1}
{ \draw[gray, thick] (32,\j*16) -- (32,16+\j*16) -- (16,16+\j*16) -- (16,\j*16) --(32,\j*16);
}
\foreach \j in {0,...,3}
{
    \draw[gray, thick] (16,\j*8) -- (16,8+\j*8) -- (8,8+\j*8) -- (8,\j*8) --(16,\j*8);
}
\foreach \j in {0,...,7}
{
    \draw[gray, thick] (8,\j*4) -- (8,4+\j*4) -- (4,4+\j*4) -- (4,\j*4) --(8,\j*4);
}
\foreach \j in {0,...,15}
{
    \draw[gray, thick] (4,\j*2) -- (4,2+\j*2) -- (2,2+\j*2) -- (2,\j*2) --(4,\j*2);
}
\foreach \j in {0,...,31}
{
    \draw[gray, thick] (2,\j*1) -- (2,1+\j*1) -- (1,1+\j*1) -- (1,\j*1) --(2,\j*1);
}
\foreach \j in {0,...,63}
{
    \draw[gray, thick] (1,\j*1/2) -- (1,1/2+\j*1/2) -- (1/2,1/2+\j*1/2) -- (1/2,\j*1/2) --(1,\j*1/2);
}

\foreach \j in {0,...,127}
{
    \draw[gray, thick] (1/2,\j*1/4) -- (1/2,1/4+\j*1/4) -- (1/4,1/4+\j*1/4) -- (1/4,\j*1/4) --(1/2,\j*1/4);
}
\foreach \j in {0,...,255}
{
    \draw[gray, thick] (1/4,\j*1/8) -- (1/4,1/8+\j*1/8) -- (1/8,1/8+\j*1/8) -- (1/8,\j*1/8) --(1/4,\j*1/8);
}

\node at (0.25,-1.5){$0$};
\node at (32,-1.5){$1$};
\node at (64,-1.5){$2$};
\node at (0.25,33.5){${\mathbf i}$};
\end{tikzpicture}
\end{tabular}
\vskip -0.25cm
\caption{(left): the $E$-continuum $K$; (right): the only non-degenerate atom.}\label{negative}
\vskip -0.25cm
\end{figure}
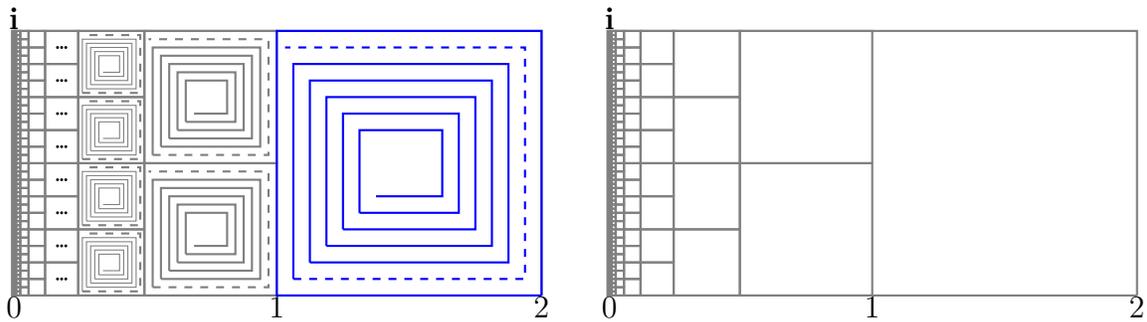
Let $f_1(z)=\frac{z}{2}$ and $f_2(z)=\frac{z+{\mathbf i}}{2}$. Let $K_0=\overline{Y}$. For all $n\ge1$, let $K_n=f_1\left(K_{n-1}\right)\cup f_2\left(K_{n-1}\right)$. Then $K_0,K_1,\ldots$ is an infinite sequence of continua converging to the segment $[0,{\mathbf i}]$ under Hausdorff distance. Clearly,
\[K=\left(\bigcup\limits_{n\ge0}K_n\right)\cup\{s{\mathbf i}: 0\le s\le1\}\]
is an $E$-continuum. See left part of Figure \ref{negative}.  Let $L_0=\partial X$. For all $n\ge1$, let $L_n=f_1\left(L_{n-1}\right)\cup f_2\left(L_{n-1}\right)$. Then $L_0,L_1,\ldots$ is an infinite sequence of continua converging to the segment $[0,{\mathbf i}]$ under Hausdorff distance. Similarly, we see that
\[L=\left(\bigcup\limits_{n\ge0}L_n\right)\cup\{s{\mathbf i}: 0\le s\le1\}\]
is also an $E$-continuum. See right part of Figure \ref{negative}. Moreover, the continuum $K$ has exactly one atom of order $1$ that is not a singleton. This atom equals $L$. Thus we have
\[
\lambda_K(x)=\left\{\begin{array}{ll}1&x\in L\\ 0& {otherwise}\end{array}\right.\ \text{and}\
\tilde{\lambda}_K(x)=\left\{\begin{array}{ll} %0& x\in[0,{\mathbf i}]\\
1& x\in L\setminus[0,{\mathbf i}]\\ 0& {otherwise}.\end{array}\right.
% \ \text{and} \  \lambda_K(x)-\tilde{\lambda}_K(x)=\left\{\begin{array}{ll} 1& x\in [0,{\mathbf i}]\\ 0& \text{otherwise}.\end{array}\right.
\]
\end{exam}

\begin{exam}\label{finite-comp}
Let $\mathcal{C}$ denote Cantor's ternary set. Let $U_1\subset\hat{\mathbb{C}}$ be the domain, not containing $\infty$, whose boundary consists of  $[0,1]\times{\bf i}\mathcal{C}=\{t+s{\bf i}: 0\le t\le 1, s\in\mathcal{C}\}$ and  $\partial\left([0,\frac43]\times[0,{\bf i}]\right)$.
%See the left part of Figure \ref{finite-comp-pic} for a depiction of $\partial U_1$.
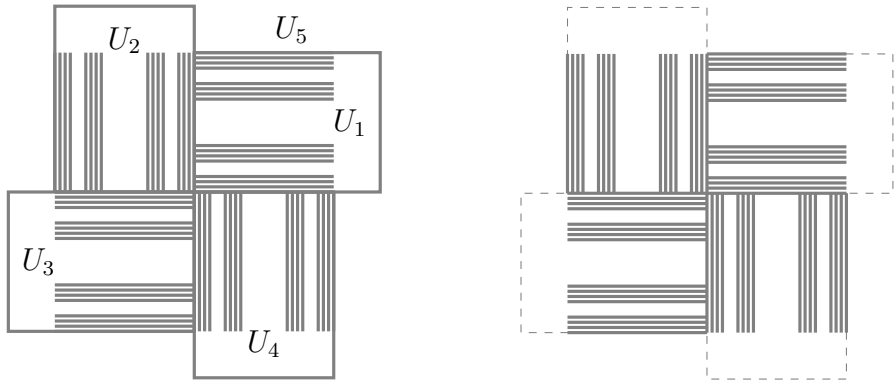
\begin{figure}[ht]
\vspace{-0.05cm}
\begin{center}
\begin{tabular}{ccc}
\begin{tikzpicture}[x=1cm,y=1cm,scale=0.618]
\draw[gray,very thick] (0,0) -- (3,0)-- (3,4)-- (0,4)-- (0,0);
\draw[gray,very thick] (3,0) -- (-1,0)-- (-1,-3)-- (3,-3)-- (3,0);
\draw[gray,very thick] (3,0) -- (3,-4)-- (6,-4)-- (6,0)-- (3,0);
\draw[gray,very thick] (3,0) -- (7,0)-- (7,3)-- (3,3)-- (3,0);
\foreach \i in {0,...,3}
{
    \draw[gray,very thick] (3*\i/27,0) -- (3*\i/27,3);
    \draw[gray,very thick] (6/9+3*\i/27,0) -- (6/9+3*\i/27,3);
     \draw[gray,very thick] (2+3*\i/27,0) -- (2+3*\i/27,3);
    \draw[gray,very thick] (2+6/9+3*\i/27,0) -- (2+6/9+3*\i/27,3);
}

\foreach \i in {0,...,3}
{
    \draw[gray,very thick] (0,3*\i/27-3) -- (3,3*\i/27-3);
    \draw[gray,very thick] (0,6/9+3*\i/27-3) -- (3,6/9+3*\i/27-3);
     \draw[gray,very thick] (0,2+3*\i/27-3) -- (3,2+3*\i/27-3);
    \draw[gray,very thick] (0,2+6/9+3*\i/27-3) -- (3,2+6/9+3*\i/27-3);
}

\foreach \i in {0,...,3}
{
    \draw[gray,very thick] (3+3*\i/27,0) -- (3+3*\i/27,-3);
    \draw[gray,very thick] (3+6/9+3*\i/27,0) -- (3+6/9+3*\i/27,-3);
     \draw[gray,very thick] (3+2+3*\i/27,0) -- (3+2+3*\i/27,-3);
    \draw[gray,very thick] (3+2+6/9+3*\i/27,0) -- (3+2+6/9+3*\i/27,-3);
}

\foreach \i in {0,...,3}
{
    \draw[gray,very thick] (3,3*\i/27) -- (3+3,3*\i/27);
    \draw[gray,very thick] (3,6/9+3*\i/27) -- (3+3,6/9+3*\i/27);
     \draw[gray,very thick] (3,2+3*\i/27) -- (3+3,2+3*\i/27);
    \draw[gray,very thick] (3,2+6/9+3*\i/27) -- (3+3,2+6/9+3*\i/27);
}

\draw(1.5,2.75) node[above]{$U_2$};
\draw(0.25,-1.5) node[left]{$U_3$};
\draw(4.5,-2.75) node[below]{$U_4$};
\draw(5.75,1.5) node[right]{$U_1$};
\draw(4.5,3.5) node[right]{$U_5$};
\end{tikzpicture}
&&\hskip 1.0cm
\begin{tikzpicture}[x=1cm,y=1cm,scale=0.618]
\draw[gray,dashed] (0,0) -- (3,0)-- (3,4)-- (0,4)-- (0,0);
\draw[gray,dashed] (3,0) -- (-1,0)-- (-1,-3)-- (3,-3)-- (3,0);
\draw[gray,dashed] (3,0) -- (3,-4)-- (6,-4)-- (6,0)-- (3,0);
\draw[gray,dashed] (3,0) -- (7,0)-- (7,3)-- (3,3)-- (3,0);

\foreach \i in {0,...,3}
{
    \draw[gray,very thick] (3*\i/27,0) -- (3*\i/27,3);
    \draw[gray,very thick] (6/9+3*\i/27,0) -- (6/9+3*\i/27,3);
     \draw[gray,very thick] (2+3*\i/27,0) -- (2+3*\i/27,3);
    \draw[gray,very thick] (2+6/9+3*\i/27,0) -- (2+6/9+3*\i/27,3);
}

\foreach \i in {0,...,3}
{
    \draw[gray,very thick] (0,3*\i/27-3) -- (3,3*\i/27-3);
    \draw[gray,very thick] (0,6/9+3*\i/27-3) -- (3,6/9+3*\i/27-3);
     \draw[gray,very thick] (0,2+3*\i/27-3) -- (3,2+3*\i/27-3);
    \draw[gray,very thick] (0,2+6/9+3*\i/27-3) -- (3,2+6/9+3*\i/27-3);
}

\foreach \i in {0,...,3}
{
    \draw[gray,very thick] (3+3*\i/27,0) -- (3+3*\i/27,-3);
    \draw[gray,very thick] (3+6/9+3*\i/27,0) -- (3+6/9+3*\i/27,-3);
     \draw[gray,very thick] (3+2+3*\i/27,0) -- (3+2+3*\i/27,-3);
    \draw[gray,very thick] (3+2+6/9+3*\i/27,0) -- (3+2+6/9+3*\i/27,-3);
}

\foreach \i in {0,...,3}
{
    \draw[gray,very thick] (3,3*\i/27) -- (3+3,3*\i/27);
    \draw[gray,very thick] (3,6/9+3*\i/27) -- (3+3,6/9+3*\i/27);
     \draw[gray,very thick] (3,2+3*\i/27) -- (3+3,2+3*\i/27);
    \draw[gray,very thick] (3,2+6/9+3*\i/27) -- (3+3,2+6/9+3*\i/27);
}
\end{tikzpicture}
\end{tabular}
\end{center}
\vskip -0.5cm
\caption{The continuum $K$ and the only non-degenerate atom $d\in\Dc_K$.}\label{finite-comp-pic}
\end{figure}
For $2\le j\le 4$ let $U_j=f^{j-1}(U_1)$, where $f(z)={\bf i}z$. See the left part of Figure \ref{finite-comp-pic}. Then $K=\bigcup_i\partial U_i$ is a continuum, whose complementary components are $U_1,\ldots, U_4, U_5$. Here $U_5$ is the one containing $\infty$. Moreover, the only non-degenerate atom of $K$ is
$\displaystyle d=\bigcup_{j=0}^3 f^j([0,1]\times{\bf i}\mathcal{C})$.
Since the continuum $d$ has a single atom, which is itself, we have
$\lambda_K(x)=\left\{\begin{array}{ll}\infty& x\in d\\ 0&{otherwise}.\end{array}\right.$
On the other hand, by the construction of $U_1,\ldots, U_4$ and $U_5$, we also have
$\tilde{\lambda}_K(x)=\left\{\begin{array}{ll}1& x\in d\\ 0&{otherwise}.\end{array}\right.$
Consequently, we have $\lambda_K(x)-\tilde{\lambda}_K(x)=\left\{\begin{array}{ll} \infty& x\in d\\ 0& {otherwise}.\end{array}\right.$

\end{exam}

%\subsection{Comparing $\lambda_K$ with $\lambda_{\partial K}$ for a compactum $K\subset\hat{\bbC}$}\label{k_vs_bd}

Then we continue to find planar continua $K$, trying to describe possible relations between $\lambda_K$ and $\lambda_{\partial K}$. The first one is Peano continuum $K$ but its boundary is a continuum that is not locally connected. Therefore, $\lambda_{\partial K}(x)\ge\lambda_K(x)$ for all $x\in\hat{\mathbb{C}}$ and $\lambda_{\partial K}(x)>\lambda_K(x)$ for uncountably many $x$.

\begin{exam}\label{bd_larger}
Consider a spiral made of broken lines, lying in the open square $W=\{t+s{\mathbf i}: 0< t,s<1\}\subset\hat{\mathbb{C}}$, which converges to $\partial W$. See  the left of Figure \ref{spiral}
\begin{figure}[ht]
\vspace{-0.05cm}
\center{\begin{tabular}{ccc}
\begin{tikzpicture}[x=0.2cm,y=0.2cm,scale=0.6]

\draw[blue, thick] (-16,0) -- (16,0) -- (16,32) -- (-16,32) --(-16,0);
\foreach \j in {1,2}
\foreach \k in {1,2}
{
\draw[blue, ultra  thick] (-4,-4*\j+16) -- (4*\j,-4*\j+16) -- (4*\j,4*\j+16) --(-4*\j-4,4*\j+16)-- (-4*\j-4,-4*\j+12) --(0,-4*\j+12);
 }
\draw[blue,ultra thick, dashed] (0,4) -- (13.0,4);

\draw[blue,ultra thick, dashed] (12.75,4.0) -- (12.75,27.5);

\end{tikzpicture}
\hspace{0.25cm}
&
\hspace{0.25cm}
\begin{tikzpicture}[x=0.2cm,y=0.2cm,scale=0.6]

% hidden = blueprint of the spiral, to be thickened

%\draw[blue,thick] (-16,0) -- (16,0) -- (16,32) -- (-16,32) --(-16,0);

%\foreach \j in {1,2} \foreach \k in {1,2} { \draw[blue, ultra  thick] (-4,-4*\j+16) -- (4*\j,-4*\j+16) -- (4*\j,4*\j+16) --(-4*\j-4,4*\j+16)-- (-4*\j-4,-4*\j+12) --(0,-4*\j+12); \draw[blue, ultra  thick] (-4,-4*\j+16-\k*0.45) -- (4*\j+\k*0.45,-4*\j+16-\k*0.45) -- (4*\j+\k*0.45,4*\j+16+\k*0.45) --(-4*\j-4-\k*0.45,4*\j+16+\k*0.45)-- (-4*\j-4-\k*0.45,-4*\j+12-\k*0.45) --(0,-4*\j+12-\k*0.45);}

%\draw[blue,ultra thick, dashed] (0,4) -- (13.2,4); \draw[blue,ultra thick, dashed] (0,3.55) -- (13.2,3.55); \draw[blue,ultra thick, dashed] (0,3.1) -- (13.2,3.1);

%\draw[blue,ultra thick, dashed] (12.5,3.0) -- (12.5,27.5); \draw[blue,ultra thick, dashed]  (12.75,3.0) -- (12.75,27.5); \draw[blue,ultra thick, dashed] (13,3.0) -- (13,27.5);

\fill[gray!80] (-16,0) -- (16,0) -- (16,32) -- (-16,32) --(-16,0);
\draw[blue, thick] (-16,0) -- (16,0) -- (16,32) -- (-16,32) --(-16,0);
\foreach \j in {1,2}
\foreach \k in {1,2}
{
\draw[gray!16, ultra  thick] (-4,-4*\j+16) -- (4*\j,-4*\j+16) -- (4*\j,4*\j+16) --(-4*\j-4,4*\j+16)-- (-4*\j-4,-4*\j+12) --(0,-4*\j+12);
\draw[gray!16, ultra  thick] (-4,-4*\j+16-\k*0.45) -- (4*\j+\k*0.45,-4*\j+16-\k*0.45) -- (4*\j+\k*0.45,4*\j+16+\k*0.45) --(-4*\j-4-\k*0.45,4*\j+16+\k*0.45)-- (-4*\j-4-\k*0.45,-4*\j+12-\k*0.45) --(0,-4*\j+12-\k*0.45);
 }

\draw[gray!16,ultra thick, dashed] (0,4) -- (13.2,4); \draw[gray!16,ultra thick, dashed] (0,3.55) -- (13.2,3.55); \draw[gray!16,ultra thick, dashed] (0,3.1) -- (13.2,3.1);

% \draw[gray!16,ultra thick, dashed] (12.5,3.0) -- (12.5,27.5); \draw[gray!16,ultra thick, dashed] (12.75,3.0) -- (12.75,27.5); \draw[gray!16,ultra thick, dashed] (13,3.0) -- (13,27.5);

\end{tikzpicture}
\hspace{0.25cm}
&
\hspace{0.25cm}
\begin{tikzpicture}[x=0.2cm,y=0.2cm,scale=0.6]
\fill[gray!80] (-16,0) -- (16,0) -- (16,32) -- (-16,32) --(-16,0);
\draw[blue, thick] (-16,0) -- (16,0) -- (16,32) -- (-16,32) --(-16,0);
\foreach \j in {1,2}
\foreach \k in {1,2}
{
\draw[gray!16, ultra  thick] (-4,-4*\j+16) -- (4*\j,-4*\j+16) -- (4*\j,4*\j+16) --(-4*\j-4,4*\j+16)-- (-4*\j-4,-4*\j+12) --(0,-4*\j+12);
\draw[gray!16, ultra  thick] (-4,-4*\j+16-\k*0.45) -- (4*\j+\k*0.45,-4*\j+16-\k*0.45) -- (4*\j+\k*0.45,4*\j+16+\k*0.45) --(-4*\j-4-\k*0.45,4*\j+16+\k*0.45)-- (-4*\j-4-\k*0.45,-4*\j+12-\k*0.45) --(0,-4*\j+12-\k*0.45);
 }
\draw[gray!16,ultra thick, dashed] (0,3.8) -- (13.2,3.8);
\draw[gray!16,ultra thick, dashed] (0,3.55) -- (13.2,3.55);
\draw[gray!16,ultra thick, dashed] (0,3.1) -- (13.2,3.1);

%\draw[gray!16,ultra thick, dashed] (12.5,3.0) -- (12.5,27.5); \draw[gray!16,ultra thick, dashed] (12.75,3.0) -- (12.75,27.5); \draw[gray!16,ultra thick, dashed] (13,3.0) -- (13,27.5);

\foreach \j in {1,2}
\foreach \k in {2}
{
\draw[blue] (-4,-4*\j+16+0.2) -- (4*\j-0.2,-4*\j+16+0.2) -- (4*\j-0.2,4*\j+16-0.2) --(-4*\j-4+0.2,4*\j+16-0.2)-- (-4*\j-4+0.2,-4*\j+12+0.2) --(0,-4*\j+12+0.2);
\draw[blue] (-4,-4*\j+16-\k*0.45-0.3) -- (4*\j+\k*0.45+0.3,-4*\j+16-\k*0.45-0.3) -- (4*\j+\k*0.45+0.3,4*\j+16+\k*0.45+0.3) --(-4*\j-4-\k*0.45-0.3,4*\j+16+\k*0.45+0.3)-- (-4*\j-4-\k*0.45-0.3,-4*\j+12-\k*0.45-0.3) --(0,-4*\j+12-\k*0.45-0.3);
 }

% division bars
\foreach \i in {0,1,2}
{
\draw[blue] (-4+3.9*\i,12.25) -- (-4+3.9*\i,12-1.25);
\draw[blue] (3.8,13.8+3*\i) -- (5.2,13.8+3*\i);
\draw[blue] (-7.8,13.8+3*\i) -- (-9.2,13.8+3*\i);
% further lower bars
\draw[blue] (-7.8,12-1.9*\i) -- (-9.2,12-1.9*\i);
}

\foreach \i in {0,...,3}
{
\draw[blue] (-5.6+3.6*\i,19.8) -- (-5.6+3.6*\i,21.2);
\draw[blue] (-7.8+1.2*\i,8.2) -- (-7.8+1.2*\i,6.8);
\draw[blue] (-3+1.2*\i,8.2) -- (-3+1.2*\i,6.8);
\draw[blue] (1.8+1.0*\i,8.2) -- (1.8+1.0*\i,6.8);
\draw[blue] (4.8+1.0*\i,8.2) -- (4.8+1.0*\i,6.8);
% horizontal bars below
\draw[blue] (7.8,7.8+1.0*\i) -- (9.2,7.8+1.0*\i);
\draw[blue] (7.8,11.8+1.0*\i) -- (9.2,11.8+1.0*\i);
\draw[blue] (7.8,15.8+0.8*\i) -- (9.2,15.8+0.8*\i);
\draw[blue] (7.8,19.2+0.8*\i) -- (9.2,19.2+0.8*\i);
\draw[blue] (7.8,22.4+0.7*\i) -- (9.2,22.4+0.7*\i);
% vertical bars below
\draw[blue] (7.8-0.8*\i,23.8) -- (7.8-0.8*\i,25.2);
\draw[blue] (4.6-0.8*\i,23.8) -- (4.6-0.8*\i,25.2);
\draw[blue] (1.4-0.8*\i,23.8) -- (1.4-0.8*\i,25.2);
\draw[blue] (-1.8-0.8*\i,23.8) -- (-1.8-0.8*\i,25.2);
\draw[blue] (-5-0.8*\i,23.8) -- (-5-0.8*\i,25.2);
\fill[blue!62](-9-\i,24.5) circle(1.8pt);
}

\end{tikzpicture}

\end{tabular}
}
\caption{A Peano continuum $K$ whose boundary is not locally connected.}\label{spiral}
\end{figure}
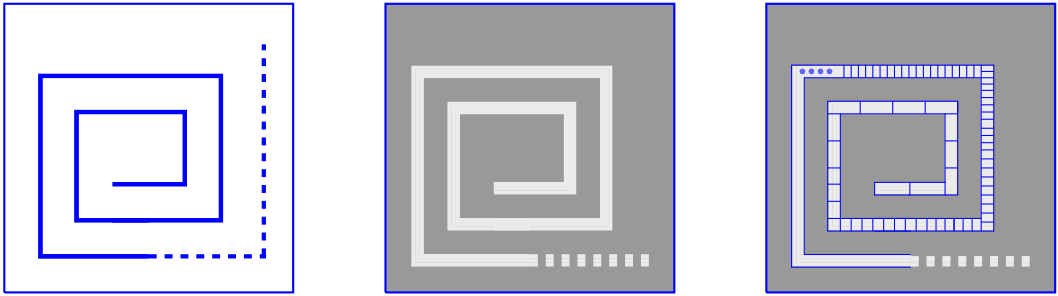
We may thicken the spiral to an embedding $h: [0,\infty)\times[0,1]\rightarrow W$, of the unbounded strip $U=[0,\infty)\times[0,1]$. Such an embedding may be chosen appropriately, so that $h(\partial U)$ consists of countably many segments. Then, we obtain a continuum
$K_0=\overline{W}\setminus h(U)$. See the middle part of Figure \ref{spiral}. Clearly, the continuum $K_0$ is not locally connected on $\partial W$; and it is locally connected at all the other points. Now, divide the thickened spiral $h(U)$ into smaller and smaller quadruples, which are depicted in the right part of Figure \ref{spiral} as small rectangles. Let $K$ be the union of $K_0$ with the newly added bars, used in the above mentioned division. Then $K$ is locally connected everywhere hence is a Peano continuum. However, its boundary $\partial K$ is not locally connected on $\partial W$ and is locally connected elsewhere. Therefore, we have
\[\lambda_K(x)\equiv 0\quad{and}\quad
\lambda_{\partial K}(x)=\left\{\begin{array}{ll}1 & x\in\partial W\\
0& x\notin\partial W.\end{array}\right.\]
\end{exam}

\begin{exam}\label{bd_smaller}
Let the continuum $K$ be defined as in Example \ref{bd_larger}. Let $f_j(z)=\frac{z}{2}+\frac{j-1}{2}{\mathbf i}$ for $j=1,2$. For any compact set $X\subset\hat{\bbC}$, put $\Phi(X)=f_1(X)\cup f_2(X)$. We will use the continuum $K$ and the mapping $\Phi$ to construct a continuum $L$. See Figure \ref{spiral-double}.
\begin{figure}[ht]
\vskip -0.05cm
\center{
\begin{tikzpicture}[x=0.2cm,y=0.2cm,scale=0.8]
\draw[gray,thick] (32,0) -- (0,0) -- (0,32) -- (32,32);
\draw[gray,thick] (32,0) -- (64,0) -- (64,32) -- (32,32) -- (32,0);
\foreach \j in {0,1}
{ \draw[gray, thick] (32,16+\j*16) -- (16,16+\j*16) -- (16,\j*16) --(32,\j*16);
}
\foreach \j in {0,...,3}
{
    \draw[gray, thick] (16,\j*8) -- (16,8+\j*8) -- (8,8+\j*8) -- (8,\j*8) --(16,\j*8);
}
\foreach \j in {0,...,7}
{
    \draw[gray, thick] (8,\j*4) -- (8,4+\j*4) -- (4,4+\j*4) -- (4,\j*4) --(8,\j*4);
}
\foreach \j in {0,...,15}
{
    \draw[gray, thick] (4,\j*2) -- (4,2+\j*2) -- (2,2+\j*2) -- (2,\j*2) --(4,\j*2);
}
\foreach \j in {0,...,31}
{
    \draw[gray, thick] (2,\j*1) -- (2,1+\j*1) -- (1,1+\j*1) -- (1,\j*1) --(2,\j*1);
}
\foreach \j in {0,...,63}
{
    \draw[gray, thick] (1,\j*1/2) -- (1,1/2+\j*1/2) -- (1/2,1/2+\j*1/2) -- (1/2,\j*1/2) --(1,\j*1/2);
}

\foreach \j in {0,...,127}
{
    \draw[gray, thick] (1/2,\j*1/4) -- (1/2,1/4+\j*1/4) -- (1/4,1/4+\j*1/4) -- (1/4,\j*1/4) --(1/2,\j*1/4);
}
\foreach \j in {0,...,255}
{
    \draw[gray, thick] (1/4,\j*1/8) -- (1/4,1/8+\j*1/8) -- (1/8,1/8+\j*1/8) -- (1/8,\j*1/8) --(1/4,\j*1/8);
}

% exterior

 \draw[gray, thick] (-4,-4) -- (64+4,-4) -- (64+4,32+4) --(-3,32+4)--(-3,-3);
 \draw[gray, thick] (-3,-3) -- (64+3,-3) -- (64+3,32+2) --(-2,32+2) -- (-2,-2);
  \draw[gray, thick] (-2,-2) -- (64+2,-2) -- (64+2,32+1) --(-1,32+1) -- (-1,-1);
 \draw[gray, thick, dashed] (-1,-1)--(64+1,-1)--(64+1,30);

\node at (33,2) {$1$};  \fill(32,0) circle(2pt);
\node at (63,2) {$2$};  \fill(64,0) circle(2pt);
\node at (61,30) {$2\!+\!{\mathbf i}$};  \fill(64,32) circle(2pt);

\node at (48,16) {$\partial K+1$};
\node at (24,8) {$f_1(\partial K\!+\!1)$};
\node at (24,24) {$f_2(\partial K\!+\!1)$};

\draw[gray, very thin] (12,4) -- (-12.5,8);
\node at (-15,10) {$f_1\circ f_1(K\!+\!1)$};
\draw[gray, very thin] (12,12) -- (-12.5,16);
\node at (-15,18) {$f_1\circ f_2(K\!+\!1)$};

\draw[gray, very thin] (12,20) -- (-12.5,24);
\node at (-15,26) {$f_2\circ f_1(K\!+\!1)$};
\end{tikzpicture}
}\vskip -0.5cm
\caption{Relative locations of $\partial K+1$, $\Phi^{1}(\partial K+1)$  and $\Phi^{2}(K+1)$.}\label{spiral-double}
\end{figure}
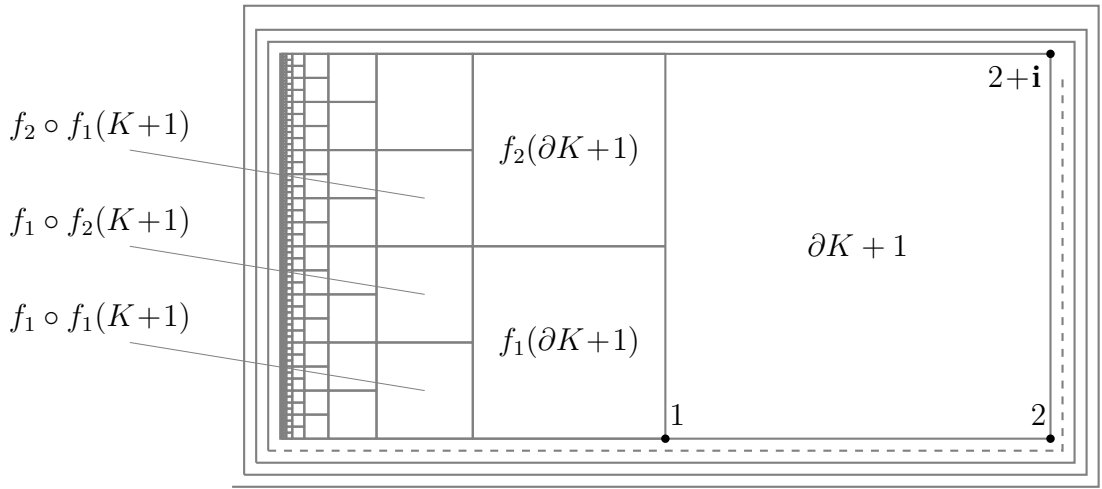
The continuum $L$ consists of five parts:
\begin{enumerate}
\item the segment $[0,{\mathbf i}]=\{s{\bf i}: 0\le s\le1\}$;
\item a spiral converging to the boundary of $[0,2]\times[0,{\mathbf i}]$;
\item $\partial K+1=\{z+1: z\in \partial K\}$;
\item $\Phi^{2n}(K+1)$ for all integers $n\ge1$; and
\item $\Phi^{2n-1}(\partial K+1)$ for all integers $n\ge1$.
\end{enumerate}
On the one hand, we can directly check that $L$ has a unique non-degenerate atom $d$, which consists of the following four parts: (1) the segment $[0,{\mathbf i}]$; (2) the boundary of $[1,2]\times[0,{\mathbf i}]$, denoted as $A$; (3) $\Phi^{2n-1}(A)$ for all integers $n\ge1$; and (4) the boundary of $[2^{-2n},2^{-2n+1}]\times[0,{\mathbf i}]$ for all integers $n\ge1$.
On the other hand, the boundary $\partial L$ has a unique non-degenerate atom $d^*$, which is the union of $A$, the segment $[0,{\mathbf i}]$, and $\Phi^{n}(A)$ for all integers $n\ge1$. See Figure \ref{atoms} for a depiction of  $d$ and $d^*$.
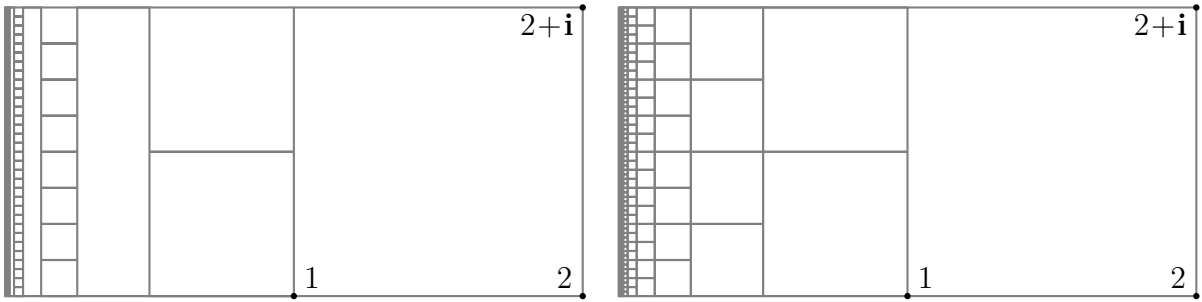
\begin{figure}[ht]
\vspace{-0.25cm}
\center{\begin{tabular}{cc}
\begin{tikzpicture}[x=0.2cm,y=0.2cm,scale=0.6]
\draw[gray,thick] (32,0) -- (0,0) -- (0,32) -- (32,32);
\draw[gray,thick] (32,0) -- (64,0) -- (64,32) -- (32,32) -- (32,0);

\foreach \j in {0,1}
{ \draw[gray, thick] (32,\j*16) -- (16,\j*16) -- (16,\j*16) --(32,\j*16);
}

%%% 1
\foreach \j in  {0}  %{0,...,3}
{
    \draw[gray, thick] (16,+\j*8) -- (16,32+\j*8) -- (8,32+\j*8) -- (8,\j*8) --(16,\j*8);
}

\foreach \j in {0,...,7}
{
    \draw[gray, thick] (8,\j*4) -- (8,\j*4) -- (4,\j*4) -- (4,\j*4) --(8,\j*4);
}

%%% 2
\foreach \j in {0}  %{0,...,15}
{
    \draw[gray, thick] (4,\j*2) -- (4,32+\j*2) -- (2,32+\j*2) -- (2,\j*2) --(4,\j*2);
}

\foreach \j in {0,...,31}
{
    \draw[gray, thick] (2,\j*1) -- (2,\j*1) -- (1,\j*1) -- (1,\j*1) --(2,\j*1);
}

%%% 3
\foreach \j in {0}  % {0,...,63}
{
    \draw[gray, thick] (1,\j*1/2) -- (1,32+\j*1/2) -- (1/2,32+\j*1/2) -- (1/2,\j*1/2) --(1,\j*1/2);
}

\foreach \j in {0,...,127}
{
    \draw[gray, thick] (1/2,\j*1/4) -- (1/2,1/4+\j*1/4) -- (1/4,1/4+\j*1/4) -- (1/4,\j*1/4) --(1/2,\j*1/4);
}

\foreach \j in {0,...,255}
{
    \draw[gray, thick] (1/4,\j*1/8) -- (1/4,1/8+\j*1/8) -- (1/8,1/8+\j*1/8) -- (1/8,\j*1/8) --(1/4,\j*1/8);
}

\node at (34,2) {$1$};  \fill(32,0) circle(2pt);
\node at (62,2) {$2$};  \fill(64,0) circle(2pt);
\node at (60,30) {$2\!+\!{\mathbf i}$};  \fill(64,32) circle(2pt);

\end{tikzpicture}

& \begin{tikzpicture}[x=0.2cm,y=0.2cm,scale=0.6]
\draw[gray,thick] (32,0) -- (0,0) -- (0,32) -- (32,32);
\draw[gray,thick] (32,0) -- (64,0) -- (64,32) -- (32,32) -- (32,0);
\foreach \j in {0,1}
{ \draw[gray, thick] (32,16+\j*16) -- (16,16+\j*16) -- (16,\j*16) --(32,\j*16);
}
\foreach \j in {0,...,3}
{
    \draw[gray, thick] (16,\j*8) -- (16,8+\j*8) -- (8,8+\j*8) -- (8,\j*8) --(16,\j*8);
}
\foreach \j in {0,...,7}
{
    \draw[gray, thick] (8,\j*4) -- (8,4+\j*4) -- (4,4+\j*4) -- (4,\j*4) --(8,\j*4);
}
\foreach \j in {0,...,15}
{
    \draw[gray, thick] (4,\j*2) -- (4,2+\j*2) -- (2,2+\j*2) -- (2,\j*2) --(4,\j*2);
}
\foreach \j in {0,...,31}
{
    \draw[gray, thick] (2,\j*1) -- (2,1+\j*1) -- (1,1+\j*1) -- (1,\j*1) --(2,\j*1);
}
\foreach \j in {0,...,63}
{
    \draw[gray, thick] (1,\j*1/2) -- (1,1/2+\j*1/2) -- (1/2,1/2+\j*1/2) -- (1/2,\j*1/2) --(1,\j*1/2);
}

\foreach \j in {0,...,127}
{
    \draw[gray, thick] (1/2,\j*1/4) -- (1/2,1/4+\j*1/4) -- (1/4,1/4+\j*1/4) -- (1/4,\j*1/4) --(1/2,\j*1/4);
}
\foreach \j in {0,...,255}
{
    \draw[gray, thick] (1/4,\j*1/8) -- (1/4,1/8+\j*1/8) -- (1/8,1/8+\j*1/8) -- (1/8,\j*1/8) --(1/4,\j*1/8);
}

\node at (34,2) {$1$};  \fill(32,0) circle(2pt);
\node at (62,2) {$2$};  \fill(64,0) circle(2pt);
\node at (60,30) {$2\!+\!{\mathbf i}$};  \fill(64,32) circle(2pt);

\end{tikzpicture}
\end{tabular}
}\vskip -0.0cm
\caption{A depiction of $d$ and $d^*$.}\label{atoms}
\end{figure}
The atom $d^*$ (of $\partial L$) is a Peano continuum and contains $d$. However, the atom $d$ (of $L$) is not locally connected at points $s{\mathbf i}$ with $0<s<1$ and is locally connected elsewhere. Therefore, we can compute their lambda functions as follows:
\[
\lambda_L(x)=\left\{\begin{array}{ll}2& x\in[0,{\mathbf i}]\\
1 & x\in d\setminus[0,{\mathbf i}]\\
0& {otherwise}\end{array}\right.\quad{and}\quad \lambda_{\partial L}(x)=\left\{\begin{array}{ll}1& x\in[0,{\mathbf i}]\\
1 & x\in d^*\setminus[0,{\mathbf i}]\\
0& {otherwise}.\end{array}\right.
\]
From these we further infer that
\[
\lambda_L(x)-\lambda_{\partial L}(x)=\left\{\begin{array}{ll} 1& x\in[0,{\mathbf i}]\\ -1& x\in d^*\setminus d\\ 0&{otherwise}.\end{array}\right.
\]
\end{exam}

Note that it is still unknown {\bf whether  there is a compactum $K\subset\hat{\mathbb C}$ such that $\lambda_K(x)\ge\lambda_{\partial K}(x)$ for all $x\in\hat{\mathbb{C}}$ and  $\lambda_K(x)>\lambda_{\partial K}(x)$ for at least one point $x\in\partial K$}.

%\subsection{On $\lambda_{X\cup Y}$ and $\lambda_{X\cap Y}$ for Peano compacta $X,Y\subset\hat{\bbC}$.}

To conclude this section, we now consider  unions and intersections of specific Peano compacta in the plane. We will find concrete Peano continua in the plane, say $X$ and $Y$, such that $X\cap Y$ is a continuum that is not locally connected. Notice that $X\cup Y$ is always a Peano continuum.

\begin{exam}\label{cap-peano}
Let $M$ be the union of $[0,1]\times\{0\}$ with the vertical segments $\{0\}\times[0,1]$ and $\{2^{-k}\}\times[0,1]$ for integers $k\ge0$. Then $M$ is a continuum and is not locally connected at points on $\{0\}\times(0,1]$; moreover, we have
\[\lambda_M(x)=\left\{\begin{array}{ll}1& x\in \{t{\bf i}: 0\le t\le 1\}\\ 0 & otherwise\end{array}\right.
\]
We will construct two Peano continua $X$ and $Y$ satisfying $X\cap Y=M$. To this end, for all integers $k\ge1$ we put
\[A_k=\bigcup_{j=1}^{2^k-1}\left[0,2^{-k+1}\right]\times\left\{j2^{-k}\right\}.\]
Then $X=M\cup\left(\bigcup_kA_k\right)$ is a Peano continuum.
\begin{figure}[ht]
\vskip -0.05cm
\begin{center}
\begin{tabular}{ccc}
\begin{tikzpicture}[x=1cm,y=1cm,scale=0.8]
\foreach \i in {1,...,3}
{
\draw[red] (0,1.296*\i) -- (2.592,1.296*\i);
}

\foreach \i in {1,...,7}
{
\draw[red] (0,0.648*\i) -- (1.296,0.648*\i);
}

\foreach \i in {1,...,15}
{
\draw[red] (0,0.324*\i) -- (0.648,0.324*\i);
}

\foreach \i in {1,...,31}
{
\draw[red] (0,0.162*\i) -- (0.324,0.162*\i);
}

\foreach \i in {1,...,63}
{
\draw[red] (0,0.081*\i) -- (0.162,0.081*\i);
}

\draw[blue,thick] (0,0) -- (0,5.184);
\draw[blue,thick] (0,0) -- (5.184,0);
\draw[red] (2.592,2.592) -- (5.184,2.592);

\foreach \i in {1,2,4,8,16,32}
{
\draw[blue,thick] (5.184/\i,0) -- (5.184/\i,5.184);
}

\fill[black] (5.184,0) circle (0.35ex); \draw[purple] (5.184,0) node[right]{$1$};
\fill[black] (5.184,5.184) circle (0.35ex); \draw[purple] (5.184,5.184) node[right]{$1+{\bf i}$};
\fill[black] (0,0) circle (0.35ex); \draw[purple] (0,0) node[left]{$0$};
\end{tikzpicture}\hspace{0.25cm}
&
&\hspace{0.25cm}
\begin{tikzpicture}[x=1cm,y=1cm,scale=0.8]
\foreach \i in {1,...,2}
{
\draw[red] (0,1.728*\i) -- (5.184,1.728*\i);
}

\foreach \i in {1,...,8}
{
\draw[red] (0,0.576*\i) -- (2.592,0.576*\i);
}

\foreach \i in {1,...,26}
{
\draw[red] (0,0.192*\i) -- (1.296,0.192*\i);
}

\foreach \i in {1,...,80}
{
\draw[red] (0,0.064*\i) -- (0.648,0.064*\i);
}

\draw[blue,thick] (0,0) -- (0,5.184);
\draw[blue,thick] (0,0) -- (5.184,0);

\foreach \i in {1,2,4,8}
{
\draw[blue,thick] (5.184/\i,0) -- (5.184/\i,5.184);
}

\fill[black] (5.184,0) circle (0.35ex);       \draw[purple] (5.184,0) node[right]{$1$};
\fill[black] (5.184,5.184) circle (0.35ex);   \draw[purple] (5.184,5.184) node[right]{$1+{\bf i}$};
\fill[black] (0,0) circle (0.35ex);           \draw[purple] (0,0) node[left]{$0$};
\end{tikzpicture}
\end{tabular}
\end{center}
\vskip -0.5cm
\caption{\small Two Peano continua that intersect at a non-locally connected continuum.}\label{peano-cap}
\end{figure}
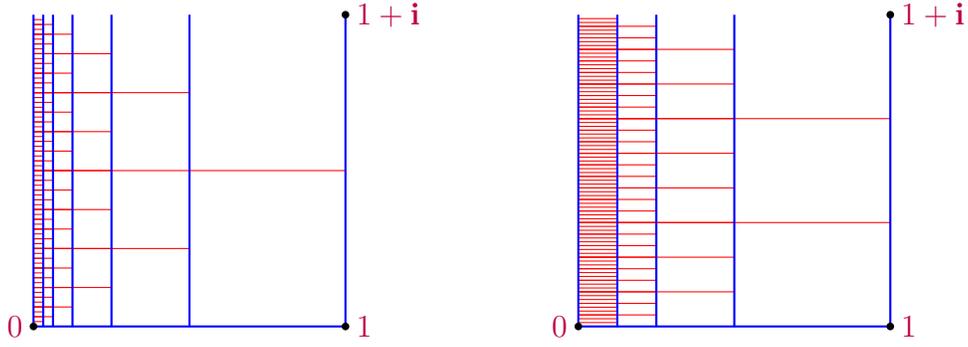
See left part of Figure \ref{peano-cap} for a rough approximate of $X$. Similarly, if for every $k\ge1$ we set
\[B_k=\bigcup_{j=1}^{3^k-1}\left[0,2^{-k+1}\right]\times\left\{j3^{-k}\right\},\]
then $\displaystyle Y=M\cup\left(\bigcup\limits_kB_k\right)$ is also a Peano continuum. See right part of Figure \ref{peano-cap} for a rough approximate of $Y$. Moreover, we have $X\cap Y=M$.
\end{exam}

We also find Peano compacta $X$ and $Y$ such that the union $X\cup Y$ is not a Peano compactum, although the intersection $X\cap Y$ is always a Peano compactum. We will use {\em fractal squares} to construct two such compacta.
Here a {\bf fractal square of order $n\ge2$} is the attractor of an iterated function system
$\displaystyle \Fc_\Dc:=\left\{f_d(x)=\frac{x+d}{n}: d\in\Dc\right\}$
for some $\Dc\subset\{0,1,\ldots,n-1\}^2$ which contains at least $2$ and at most $n^2-1$ elements.
For general theory on iterated function systems, we refer to \cite{Hutchinson81}.

\begin{exam}\label{cup-fs}
Let $X$ and $Y$ be the fractal squares determined by $\Fc_{\Dc_X}$ and $\Fc_{\Dc_Y}$. Here $\Dc_X=\{(i,0): i=0,1,2\}\cup\{(0,2)\}$ and $\Dc_Y=\{(i,0): i=0,1,2\}\cup\{(1,2),(2,2)\}$. See Figure \ref{fs-cup} for relative locations of the small squares $f_d([0,1]^2)$ with $d\in\Dc_X$  and $d\in\Dc_Y$.
\begin{figure}[ht]
%\vskip -0.15cm
\begin{center}
\begin{tabular}{ccc}
\begin{tikzpicture}[x=1cm,y=1cm,scale=0.618]

\fill[purple!30] (0,0) -- (5.184,0) -- (5.184,1.728) -- (0,1.728) -- (0,0);
\fill[purple!30] (0,3.456) -- (1.728,3.456) -- (1.728,5.184) -- (0,5.184) -- (0,3.456);

\foreach \i in {0,...,3}
{
\draw[gray,thick] (0,1.728*\i) -- (5.184,1.728*\i);
\draw[gray,thick] (1.728*\i,0) -- (1.728*\i,5.184);
}
\end{tikzpicture}
&
\hskip 0.5cm
&
\begin{tikzpicture}[x=1cm,y=1cm,scale=0.618]

\fill[purple!30] (0,0) -- (5.184,0) -- (5.184,1.728) -- (0,1.728) -- (0,0);
\fill[purple!30] (1.728,3.456) -- (5.184,3.456) -- (5.184,5.184) -- (1.728,5.184) -- (1.728,3.456);

\foreach \i in {0,...,3}
{
\draw[gray,thick] (0,1.728*\i) -- (5.184,1.728*\i);
\draw[gray,thick] (1.728*\i,0) -- (1.728*\i,5.184);
}
\end{tikzpicture}
\end{tabular}
\end{center}
\vskip -0.5cm
\caption{\small The small squares $f_d([0,1]^2)$ for $d\in\Dc_1$ (left part) and for $d\in\Dc_2$ (right part).}\label{fs-cup}
\end{figure}
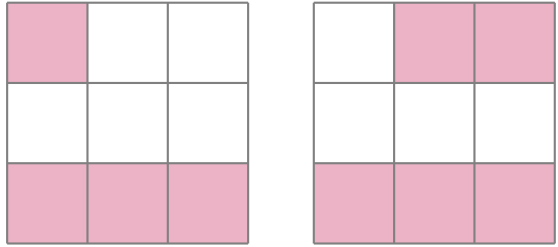
Then $X$ and $Y$ are Peano compacta,  each of which contains the interval $[0,1]$, such that $X\cup Y$ contains all the segments $[0,1]\times\{\frac{2}{3^k}\}$ for $k\ge1$. Moreover,  $X\cap Y$ is a Peano compactum having uncountably many components. All but one of these components are single points. The only non-degenerate component is the interval $[0,1]$. On the other hand, for all $k\ge1$ the horizontal strip $\bbR\times\left(\frac{1}{3^k},\frac{2}{3^k}\right)$ is disjoint from $X\cup Y$. This implies that $X\cup Y$ is not a Peano compactum. Consequently, we have
\[
\begin{array}{ccc}\lambda_{X}(x)=0 (\forall x\in \hat{\mathbb{C}}); & \lambda_{Y}(x)= 0 (\forall x\in \hat{\mathbb{C}});  & \lambda_{X\cup Y}(x)=\left\{\begin{array}{ll}1& x\in [0,1]\\ 0 & otherwise.\end{array}\right.
\end{array}
\]
Notice that $Y\cup X+1$ is also a Peano compactum, although $Y\cap X+1$ has uncountably many components. Thus  $\lambda_{X+1}(x)=\lambda_{Y}(x)=\lambda_{Y\cup X+1}(x)=0$ for all  $x\in \hat{\mathbb{C}}$. \end{exam}

\noindent
{\bf Acknowledgement}. The authors are grateful to Dr. Yi Yang at Sun Yat-sen University, for valuable discussions during the period of his phd study.
%In particular, we have shared many experiences in computing the lambda function for concrete planar compacta, such as a polynomial Julia set that is renormalizable. Further questions of such a nature are to be discussed in a forthcoming paper.

%%% bibliography
\bibliographystyle{plain}
%\bibliography{biblio}

%\newpage  \input{T-Thm.tex}
\end{document}